\documentclass[letterpaper,12pt]{amsart}
\RequirePackage[top=1in,bottom=1in,left=1.35in,right=1.35in,includehead]{geometry}

\usepackage{amsmath,amsfonts, amssymb, amsthm, amscd}

\usepackage{colortbl}
\usepackage{color,float}
\usepackage{hyperref, enumitem}
\usepackage{mathtools}
\usepackage[capitalize,nameinlink,noabbrev,nosort]{cleveref}
\usepackage{graphicx}
\usepackage{tikz}
\usepackage{multirow}
\usepackage{cancel}

\usepackage[leqno]{amsmath}
\makeatletter
\newcommand{\leqnomode}{\tagsleft@true}
\newcommand{\reqnomode}{\tagsleft@false}
\makeatother

\usepackage{cases}
\usepackage{subfigure}
\usepackage{color}
\usepackage{latexsym, longtable}
\usepackage{comment}

\usepackage[T1]{fontenc}
\usepackage{tikz-cd}

\theoremstyle{plain}
\newtheorem{theorem}{Theorem}[section]
\newtheorem{proposition}[theorem]{Proposition}

\newtheorem{lemma}[theorem]{Lemma}
\newtheorem{cor}[theorem]{Corollary}
\newtheorem{corollary}[theorem]{Corollary}

\newtheorem{conjecture}[theorem]{Conjecture}
\theoremstyle{definition}

\newtheorem{defn}[theorem]{Definition}
\newtheorem{question}[theorem]{Question}
\newtheorem{example}[theorem]{Example}

\newtheorem{remark}[theorem]{Remark}
\numberwithin{equation}{section}
\crefname{defn}{Definition}{Definitions}

\newtheorem*{mainresult*}{Main Result}

\DeclareMathOperator{\PQT}{Tab}

\DeclareMathOperator{\PDS}{PDS}

\DeclareMathOperator{\dg}{dg}

\DeclareMathOperator{\arm}{\text{arm}}
\DeclareMathOperator{\leg}{\text{leg}}
\DeclareMathOperator{\maj}{maj}

\DeclareMathOperator{\inv}{inv}

\DeclareMathOperator{\Des}{Des}
\DeclareMathOperator{\coinv}{coinv}
\DeclareMathOperator{\ribbon}{ribbon}

\DeclareMathOperator{\quinv}{quinv}

\DeclareMathOperator{\TAZRP}{TAZRP}

\DeclareMathOperator{\South}{South}
\DeclareMathOperator{\SSYT}{SSYT}
\newcommand{\LLT}{LLT}
\newcommand{\rw}{rw}
\newcommand{\boldnu}{\widehat{\boldsymbol{\nu}}}
\newcommand{\invtwo}{\widehat{\inv}}
\newcommand{\armtwo}{\widehat{\arm}}
\newcommand{\qbinom}[2]{\bgroup\renewcommand*{\arraystretch}{1}\begin{bmatrix} #1 \\ #2\end{bmatrix} \egroup}

\newlength\cellsize 
\savebox2{%
\begin{picture}(12,12)
\put(0,0){\line(1,0){12}}
\put(0,0){\line(0,1){12}}
\put(12,0){\line(0,1){12}}
\put(0,12){\line(1,0){12}}
\end{picture}}

\newcommand\cellify[1]{\def\thearg{#1}\def\nothing{}%
\ifx\thearg\nothing
\vrule width0pt height\cellsize depth0pt\else
\hbox to 0pt{\usebox2\hss}\fi%
\vbox to 12\unitlength{
\vss
\hbox to 12\unitlength{\hss$#1$\hss}
\vss}}

\newcommand\tableau[1]{\vtop{\let\\=\cr
\setlength\baselineskip{-16000pt}
\setlength\lineskiplimit{16000pt}
\setlength\lineskip{0pt}
\halign{&\cellify{##}\cr#1\crcr}}}

\savebox3{%
\begin{picture}(12,12)
\put(0,0){\line(1,0){12}}
\put(0,0){\line(0,1){12}}
\put(12,0){\line(0,1){12}}
\put(0,12){\line(1,0){12}}
\end{picture}}

\newcommand\expath[1]{%
\hbox to 0pt{\usebox3\hss}%
\vbox to 12\unitlength{
\vss
\hbox to 12\unitlength{\hss$#1$\hss}
\vss}}

\newcommand\cell[3]{
\def\i{#1} \def\j{#2} \def\entry{#3}

\draw (\j-1,-\i)--(\j,-\i)--(\j,-\i+1)--(\j-1,-\i+1)--(\j-1,-\i);
\node at (\j-.5,-\i+.5) {\entry};
}

\newcommand{\qtrip}[3]{
\begin{tikzpicture}[scale=0.5]
\cell{1}{0}{#1} \cell{2}{0}{#2} \cell{2}{2.7}{#3}
\node at (1,-1.5) {$\cdots$};
\end{tikzpicture}
}

\newcommand\cellL[4]{
\def\i{#1} \def\j{#2} \def\entry{#3} \def\labl{#4}
\draw (\j-1,-\i)--(\j,-\i)--(\j,-\i+1)--(\j-1,-\i+1)--(\j-1,-\i);
\node at (\j-.5,-\i+.5) {\entry};
\node at (\j-1.3,-\i+.8) {\tiny\labl};
}

\newcommand\graycell[3]{
\def\i{#1} \def\j{#2} \def\entry{#3}
\draw[fill=gray!60](\j-1,-\i)--(\j,-\i)--(\j,-\i+1)--(\j-1,-\i+1)--(\j-1,-\i);
\node at (\j-.5,-\i+.5) {\entry};
}

\newcommand{\I}{\mathfrak{S}}

\newcommand{\cQ}{\mathcal{Q}}

\renewcommand{\arraystretch}{2}

\begin{document}

\title
[Modified Macdonald polynomials and the multispecies TAZRP: I]
{Modified Macdonald polynomials and the multispecies zero range process: I}

\author[A.~Ayyer]{Arvind Ayyer}
\address{A.~Ayyer, Department of Mathematics, Indian Institute of Science, Bangalore 560 012, India}
\email{arvind@iisc.ac.in}

\author[O.~Mandelshtam]{Olya Mandelshtam}
\address{O.~Mandelshtam, Department of Combinatorics and Optimization, University of Waterloo, Waterloo, ON, Canada}
\email{omandels@uwaterloo.ca}

\author[J.~B.~Martin]{James B.\ Martin}
\address{J.~B.~Martin, Department of Statistics, University of Oxford, UK}
\email{martin@stats.ox.ac.uk}

\begin{abstract}
In this paper we prove a new combinatorial formula for the modified Macdonald polynomials $\widetilde{H}_\lambda(X;q,t)$, motivated by connections to the theory of interacting particle systems from statistical mechanics. The formula involves a new statistic called \emph{queue inversions} on fillings of tableaux. This statistic is closely related to the \emph{multiline queues} which were recently used to give a formula for the Macdonald polynomials $P_\lambda(X;q,t)$. In the case $q=1$ and $X=(1,1,\dots,1)$, that formula had also been shown to compute stationary probabilities for a particle system known as the \emph{multispecies ASEP} on a ring, and it is natural to ask whether a similar connection exists between the modified Macdonald polynomials and a suitable statistical mechanics model. In a sequel to this work, we demonstrate such a connection, showing that the stationary probabilities of the \emph{multispecies totally asymmetric zero-range process} (mTAZRP) on a ring can be computed using tableaux formulas with the queue inversion statistic. This connection extends to arbitrary $X=(x_1,\dots, x_n)$; the $x_i$ play the role of site-dependent jump rates for the mTAZRP. 
\end{abstract}



\maketitle

\setcounter{tocdepth}{1}

\tableofcontents

\section{Introduction}

The theory of symmetric functions has a long history, having its origins in invariant theory, Galois theory, group theory and, of course, combinatorics. As Stanley~\cite[Notes in Chapter~7]{stanley-ec2} remarks, the first published work on symmetric functions was a derivation of the well-known Newton-Girard identity by A. Girard~\cite{girard-1884} in 1629. The reason for the name is that it was independently rediscovered by Newton around 1666. 

On the other hand, the theory of interacting particle systems is relatively modern. It was an influential paper of Spitzer~\cite{spitzer-1970} in 1970 that initiated the subject and set out the important questions in the field. In particular, the simple exclusion process and the zero-range process were first defined there. 

The focus of this paper is a particular family of symmetric functions, 
the \emph{modified Macdonald polynomials}, motivated by a new link 
to an interactive particle system known as the \emph{multispecies totally
asymmetric zero-range process}. 

Over the last couple of decades, the theory of special functions and symmetric functions have found unexpected connections to diverse interacting particle systems. The asymmetric simple exclusion process  (ASEP) has played a central role in this connection. Building on work of Uchiyama, Sasamoto, and Wadati~\cite{USW,sasamoto-1999}, Corteel and Williams~\cite{CW11} found that the partition function of the single species ASEP with open boundaries is related to moments of Askey-Wilson polynomials, and discovered combinatorial formulas for those quantities. Generalizing this work, Corteel and Williams~\cite{CW18} and also Cantini~\cite{cantini-2017} showed that the partition function of the two-species ASEP with open boundaries is related to certain moments of Koornwinder polynomials, and soon after, the second author along with Corteel and Williams~\cite{CMW-2017} found the associated combinatorial formulas. More generally, Cantini, Garbali, de Gier and Wheeler~\cite{cantini-etal-2016} found that the partition function of a multispecies variant of the ASEP with open boundaries is a specialization of a Koornwinder polynomial, but finding formulas for these quantities is still an outstanding open question.

The multispecies ASEP on a ring is another version of the ASEP that has been found to have deep connections to orthogonal polynomials. Cantini, de Gier and Wheeler~\cite{CGW-2015} showed that the stationary probabilities of the ASEP on a ring are related to the nonsymmetric Macdonald polynomials, and that the partition function of the ASEP on a circle is equal to the symmetric Macdonald polynomial $P_{\lambda}(X;q,t)$ evaluated at $x_1=\cdots=x_n=q=1$. In the totally asymmetric case (at $t=0$),  Ferrari and the third author~\cite{FM07} introduced a recursive formulation which expresses
the stationary probabilities as sums over combinatorial objects known as \emph{multiline queues}. Recently, the third author generalized these multiline queues to compute probabilities for the full ASEP~\cite{martin-2020}.  Building upon that result, the second author together with Corteel and Williams further generalized these multiline queues to incorporate the parameters $x_1,\ldots,x_n,q$ to obtain formulas for $P_{\lambda}(X;q,t)$ and the related nonsymmetric Macdonald polynomials \cite{CMW18}. These formulas interpolate between probabilities of the ASEP on a ring and the Macdonald polynomials.

There are several natural bases for symmetric polynomials. The \emph{monomial symmetric polynomials}, \emph{elementary symmetric polynomials}, \emph{complete homogeneous symmetric polynomials} and \emph{power sum symmetric polynomials} are classical. Probably the single most important family is the family of \emph{Schur polynomials}, for which there are several generalizations. 
The symmetric Macdonald polynomials are a remarkable two-parameter generalization of Schur polynomials that were introduced by Macdonald in 1987 \cite{Mac88, Macdonald}. They are indexed by partitions $\lambda$ and are denoted $P_{\lambda}(X; q,t)$, where $X=(x_1,x_2,\ldots)$ is the alphabet and $q$ and $t$ are parameters; coefficients of $P_{\lambda}(X;q,t)$ are in $\mathbb{Q}(q,t)$. They can be characterized as the unique basis for the ring of symmetric functions in $q$ and $t$ satisfying certain orthogonality and triangularity conditions. 

The modified Macdonald polynomials $\widetilde{H}_{\lambda}(X;q,t)$ defined by Garsia and Haiman \cite{garsia-haiman-1993} are a special form of the $P_{\lambda}$'s with coefficients in $\mathbb{Z}[q,t]$ obtained through a formal operation called \emph{plethysm} from a scalar multiple of $P_{\lambda}(X;q,t)$. Understanding the combinatorics of these polynomials has been a fundamental area of interest in algebraic combinatorics. The most commonly used combinatorial description of $\widetilde{H}_{\lambda}(X;q,t)$ is a tableau formula due to Haglund, Haiman, and Loehr \cite{HHL05}: 
\begin{equation}
\label{eq:HHLformula}
\widetilde{H}_{\lambda}(X;q,t) = \sum_{\sigma:\dg(\lambda) \rightarrow \mathbb{Z}} t^{\inv(\sigma)}q^{\maj(\sigma)}x^{\sigma}
\end{equation}
where the sum is over all fillings of a diagram $\dg(\lambda)$ 
whose shape is the partition $\lambda$. The \emph{inversion} statistic
$\inv(\sigma)$ and the \emph{major index} statistic are explained in Section \ref{sec:definitions}.

Recently, a different combinatorial model for these polynomials has been given by Garbali and Wheeler~\cite{garbali-wheeler-2020}. Independently, the second author together with Corteel, Haglund, Mason, and Williams found formulas for $\widetilde{H}_{\lambda}(X;q,t)$ that were based on the combinatorial interpretation of plethysm applied to multiline queues \cite{compactformula}. One such formula was a compact version of the original tableaux formula of Haglund, Haiman, and Loehr. The authors of \cite{compactformula} also conjectured a formula related
to (\ref{eq:HHLformula}) in which the inversion statistic is replaced by a variant
form which we call the \emph{queue inversion} statistic.
Our main result  is a proof of this conjectural formula:
\begin{mainresult*}
The modified Macdonald polynomial may be written as
\begin{equation}\label{eq:MainResult}
\widetilde{H}_{\lambda}(X;q,t) = \sum_{\sigma:\dg(\lambda) \rightarrow \mathbb{Z}} t^{\quinv(\sigma)}q^{\maj(\sigma)}x^{\sigma},
\end{equation}
where the definition of the queue inversion statistic $\quinv(\sigma)$ is 
again given in Section \ref{sec:definitions}.
\end{mainresult*}
Both the statistic $\inv(\sigma)$ appearing in (\ref{eq:HHLformula})
and the statistic $\quinv(\sigma)$ appearing in (\ref{eq:MainResult})
are defined in terms of the content of particular triples of cells in the filling $\sigma$. 
Even though these two results look tantalizingly similar, we do 
not know of a bijection translating between the two statistics even in very simple 
cases (for instance tableaux with just two rows). 

A particular motivation for our new representation is a link between the modified Macdonald polynomials at $q=1$ and the multispecies totally asymmetric zero-range process (TAZRP) on the ring, which we establish in a companion paper \cite{AMM2}. 
The stationary probabilities of a TAZRP on a ring with $n$ sites, with site-specific 
parameters $x_1, x_2, \dots, x_n$ and a global parameter $t$, and with
particle-types determined by a partition $\lambda$, 
can be written (suitably nomalized) as polynomials in $x_1,\dots, x_n$ and $t$ whose 
sum is $\widetilde{H}_{\lambda}(X;1,t)$. Moreover, under
the probability distribution on fillings $\sigma$ of $\dg(\lambda)$ 
proportional to the weights $t^{\quinv(\sigma)} x^\sigma$ appearing on the right-hand side
of (\ref{eq:MainResult}), the distribution of the bottom row of the tableau 
projects to the stationary distribution of the TAZRP. 
See Section \ref{sec:mTAZRP} for further details of this connection.

The plan of the article is as follows. In \cref{sec:main}, we introduce necessary definitions and notation and formally state the main result mentioned above. We also give details of the link to the multispecies totally asymmetric zero-range process. \cref{sec:background} gives some further background which will be required for the proofs. In \cref{sec:sym}, we prove the symmetry of the function defined by 
the right-hand side of (\ref{eq:MainResult}). The modified Macdonald polynomials are characterized by the three properties stated in \cref{sec:axioms}. The third property will turn out to be immediate from our definition. In \cref{sec:A1m}, we prove the first of these properties. In \cref{sec:A2m}, we prove the second property in a certain nondegenerate case. Settling the second property in the degenerate case proves to be the main hurdle. This is taken care of in \cref{sec:Phi-degenerate} using a bijection on certain tableaux that is described in \cref{sec:bijPQT}. The latter is in turn obtained via a bijection on words respecting a coinversion-type statistic, which is proved in \cref{sec:bijwords}. In \cref{sec:conclusion}, we conclude the proof of our main result and mention various related questions.

\section{Main result and related work}
\label{sec:main}

\subsection{Definitions and main result}
\label{sec:definitions}
Let $\lambda=(\lambda_1,\ldots,\lambda_k)$ be a partition. The \emph{diagram of type $\lambda$}, which we denote $\dg(\lambda)$, consists of
the cells $\{(r,i), 1\leq i\leq k, 1\leq r\leq \lambda_i\}$, which we depict using
$k$ bottom-justified columns, where the $i$'th column from left to right has $\lambda_i$ boxes. See \cref{fig:readingorder} for an illustration of $\dg((3,2,1,1))$.
The cell $(r,i)$ corresponds to the cell in the $i$'th column in the $r$'th row of $\dg(\lambda)$, where rows are labeled from bottom to top. 
We warn the reader that this is the opposite convention of labelling points in the plane using cartesian coordinates, but is somewhat standard in the literature. We will use this convention throughout the article. 

A \emph{filling} of type $(\lambda,n)$ is a function $\sigma:\dg(\lambda) \rightarrow [n]$ defined on the cells of $\dg(\lambda)$ (where $[n]=\{1,2,\dots, n\}$).
We also refer to a diagram together with a filling of it as a \emph{tableau}.
Let $\PQT(\lambda,n)$ be the set of fillings of type $(\lambda, n)$,
and $\PQT(\lambda)$ the set of functions $\sigma:\dg(\lambda) \rightarrow \mathbb{Z}^+$.

\begin{center}
\begin{figure}[h]
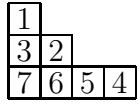

\[
\tableau{1\\3&2\\7&6&5&4}
\]
\caption{The diagram $\dg((3,2,1,1))$ with the cells labeled according to the reading order in \cref{def:readingorder}. Our convention is that the cell labelled $5$ has coordinates $(1,3)$.}
\label{fig:readingorder}
\end{figure}
\end{center}

For $\sigma$ in $\PQT(\lambda, n)$ or in $\PQT(\lambda)$, and a cell $x=(r,i)$ of $\dg(\lambda)$, we call $\sigma(x)$ the \emph{contents}
of the cell $x$ in the filling $\sigma$. 
For convenience, we also define $\sigma((\lambda_i+1,i))=0$, where $(\lambda_i+1,i)$ is the nonexistent cell above column $i$.
When $r>1$,  define $\South(x)$ to be the cell $(r-1,i)$ directly below cell $x$ in the same column. 

\begin{defn}
\label{def:readingorder}
Define the \emph{reading order} on the cells of a tableau to be along the rows from right to left,
taking the rows from top to bottom. 
\end{defn}

See \cref{fig:readingorder} for an illustration of the reading order.
We now define various statistics defined on fillings of tableaux. 

\begin{defn}\label{def:maj}
Let $\sigma\in\PQT(\lambda)$. Define the \emph{leg} of a cell $(r,i)$ to be the number of cells in column $i$ above row $r$, i.e. $\leg((r,i))= \lambda_i - r$. Define the \emph{major index} of the filling $\sigma$ to be:
\[
\maj(\sigma) = \sum_{x\,:\,\sigma(x)>\sigma(\South(x))} (\leg(x)+1)
\]
\end{defn}

\begin{defn}\label{def:quinv}
Given a diagram $\dg(\lambda)$, a \emph{triple} consists of either
\begin{itemize}
\item three cells $(r+1,i)$, $(r,i)$ and $(r,j)$ with $i<j$; or
\item two cells $(r,i)$ and $(r,j)$ with $i<j$ and  $\lambda_i = r$
(in which case the triple is called a \emph{degenerate triple}).
\end{itemize}
Let us write $a=\sigma((r+1 ,i))$, $b=\sigma((r,i))$, and $c=\sigma((r,j))$, 
for the contents of the cells of the triple, so that we can 
depict a triple along with its contents as
\begin{center}
\qtrip{$a$}{$b$}{$c$}, \qquad or \qquad \begin{tikzpicture}[scale=0.5]
\node at (-0.5,-.5) {$0$};
 \cell{2}{0}{$b$} \cell{2}{2.7}{$c$}
\node at (1,-1.5) {$\cdots$};
\end{tikzpicture}
\end{center}

We say that a triple is a \emph{queue inversion triple},
or a \emph{quinv triple} for short,
if its entries are oriented counterclockwise when the entries are read in increasing order, with ties being broken with respect to reading order. If the triple is degenerate with content $b,c$, it is a $\quinv$ triple if and only if $b<c$. 
(this is equivalent to thinking of a degenerate triple as a regular triple with $a=0$).
\end{defn}

Accordingly, define $\cQ$ to be the set of contents such that 
\begin{equation}
\label{cQdef}
\cQ=\big\{(a,b,c)\in[n]^3:
a<b<c,
\text{ or } b<c<a,
\text{ or } c<b<a,
\text{ or } a=b\ne c \big\}.
\end{equation}
Then the triple $((r+1,i),(r,i),(r,j))$ where $i<j$ forms a quinv triple in $\sigma$ if and only if $(\sigma((r+1,i)),\sigma((r,i)),\sigma((r,j)))\in \cQ$.

\begin{remark} An easy way to check if a triple is a $\quinv$ triple is: if the two entries in the same column are the same but not equal to the third one, or if all three entries are different and increasing in the counterclockwise direction, then the triple is a quinv triple. In all other cases, the triple is not a $\quinv$ triple. For example, the following are $\quinv$ triples: 
\[
\qtrip 123\ ,\quad 
\qtrip 312\ ,\quad 
\qtrip 231\ ,\quad
\qtrip 112\ , \quad
\qtrip 221\ ,
\]
whereas the following are not:
\[
\qtrip 132\ , \quad
\qtrip 121\ , \quad
\qtrip 211\ , \quad
\qtrip 122\ , \quad
\qtrip 111\ .
\]
\end{remark}

 \begin{figure}[ht!]
 \begin{tikzpicture}[scale=0.5]
\cell013\cell022
\cell111\cell123\cell131\cell143\cell153
\cell211\cell221\cell232\cell241\cell252\cell263\cell273
 \end{tikzpicture}
  \centering
  \caption{A tableau of type $\lambda=(3,3,2,2,2,1,1)$ and $n=3$. The weight of this filling is $x_1^5x_2^3x_3^6  q^5 t^{12}$.}\label{fig:polyqueue}
 \end{figure}

\begin{defn}\label{def:stats}
The \emph{weight} of a filling $\sigma$ is $x^{\sigma}t^{\quinv(\sigma)}q^{\maj(\sigma)}$, where:
\begin{itemize}
\item $x^{\sigma}$ is the monomial corresponding to the content of $\sigma$,
\item $\quinv(\sigma)$ is the total number of $\quinv$ triples in $\sigma$,
\item $\maj(\sigma)$ is the major index given in \cref{def:maj}.
 \end{itemize}
 \end{defn}
 
See \cref{fig:polyqueue} for an example of a tableau and its
weight.

The main result we present in this article is the following one, 
which was conjectured by the second author, Corteel, Haglund, Mason, and Williams.
\begin{theorem}
\label{thm:mainconj}
Let $\lambda$ be a partition. The modified Macdonald polynomial can be written as
\[
\widetilde{H}_{\lambda}(X;q,t) = \sum_{\sigma\in\PQT(\lambda)} x^{\sigma}t^{\quinv(\sigma)}q^{\maj(\sigma)}.
\]
\end{theorem}

\noindent\emph{Proof sketch.} We first show that our formula is symmetric in the variables $X$ in \cref{thm:symmetry}. The modified Macdonald polynomials $\widetilde{H}_{\lambda}(X;q,t)$ are symmetric functions that are uniquely characterized by the axioms \eqref{eq:A1}, \eqref{eq:A2}, and \eqref{eq:A3}. Our strategy is to show that our formula satisfies these axioms. Axiom \eqref{eq:A3} is immediate from our definition. Axioms \eqref{eq:A1} and \eqref{eq:A2} are written equivalently as \eqref{A1m} and \eqref{A2m}, respectively. We employ the canonical tool of superization to deal with the negative sign in the plethysm. Thus we introduce ``super fillings'' denoted by $\widetilde{\PQT}(\lambda)$ in \cref{def:super}. The axioms \eqref{A1m} and \eqref{A2m} then become equivalent to \eqref{C1} and \eqref{C2}, respectively. \cref{lem:A1FP} proves \eqref{C1} and \cref{thm:A2FP} combined with \cref{prop:Phi-degenerate} proves \eqref{C2}. See \cref{sec:conclusion} for the formal proof.

\begin{remark}
Note that when $\lambda$ has all parts distinct, one avoids the technical parts of our proof, with \cref{cor:nondegenerate} proving \eqref{C2}, and the proof is greatly simplified. In this case, all our arguments directly follow the proof strategy of \cite[Theorem 2.2]{HHL05}.
\end{remark}

\begin{remark}\label{rem:terminology}
The tableaux which we consider are closely related to the
\emph{multiline queues} mentioned in the introduction, which are used in 
\cite{FM07}, \cite{martin-2020}, \cite{CMW18} and elsewhere. 
We briefly explain the correspondence, and the term ``queue inversion''.
(The details are not required for the rest of the paper.)

Consider a partition $(\lambda_1, \dots, \lambda_k)$ and a filling $\sigma\in\PQT(\lambda,n)$. 
We consider a system composed of a sequence of queues in series,
labelled $\lambda_1, \lambda_1-1, \dots, 1$, 
and with a collection customers labelled $1,2,\dots, k$. 

Customer $i$ enters the system in queue $\lambda_i$, and 
proceeds sequentially to queues $\lambda_i-1, \dots, 1$. 
The entry $\sigma((r,i))$ represents the ``time'' that customer
$i$ is served at queue $r$ and enters queue $r-1$. 
In this sense ``time'' runs from right to left, and wraps 
cyclically around. The time that customer $i$ occupies 
is the (cyclic) interval $[\sigma((r,i)), \sigma((r+1,i))]$. 

Customers have a priority order, with customer $i$ having
higher priority than customer $j$ for $i<j$. Take any $i<j$. If the service time
$\sigma((r,j))$ of customer $j$ at queue $r$ is strictly within the
interval $[\sigma((r,i)), \sigma((r+1,i))]$ during which customer $i$ is present
at that queue, then this contradicts the priority order. 
Such an event corresponds to 
the triple of cells $(r+1,i)$, $(r,i)$, and $(r,j)$ forming a quinv triple. 

Putting $t=0$ corresponds to a case where the priority order is strictly enforced
(and gives positive weight only to tableaux containing no quinv triples). 
\end{remark}

Finally for completeness, we compare the queue inversion statistic
defined above to the inversion statistic used in Haglund, Haiman and Loehr's formula 
(\ref{eq:HHLformula}), which we refer to as the HHL inversion statistic. Whereas the queue inversion statistic 
concerns triples of cells of the form $(r+1, i), (r,i), (r,j)$, with $i<j$
(with a degenerate triple in the case $r=\lambda_i$),
the HHL inversion statistic concerns triples of the form
$(r,i), (r-1,i), (r,j)$ with $i<j$ 
(with a degenerate triple in the case $r=1$),
which we can depict as 
\[
\raisebox{-0.1in}{\qtrip {$a$}{$b$}{$c$}} \quad
\text{and}
\quad
\raisebox{-0.1in}{\begin{tikzpicture}[scale=0.5]
\cell{1}{1}{$a$} \cell{1}{3.7}{$c$} \cell{2}{1}{$b$}
\node at (2,-0.5) {$\cdots$};
\end{tikzpicture}}
\]
respectively. 
A triple of the latter kind is called an inversion triple if\linebreak[4] $(\sigma((r,i)), \sigma((r-1,i)), \sigma((r,j)))\in\cQ$
(with $\sigma(0,i)$ taken to be $\infty$ for all $i$). 
Then $\inv(\sigma)$ is the number of inversion triples in the filling $\sigma$. 

\subsection{Multispecies totally asymmetric zero range process}
\label{sec:mTAZRP}
We consider an interacting particle system which we call the \emph{multispecies totally asymmetric zero range process} (or TAZRP), which appears in \cite{takeyama-2015}, and is a specialization of a much wider class of systems known as multispecies zero range processes;
see for example the 
 review in \cite{kuniba-okado-watanabe-2017}.

Fix a partition $\lambda=(\lambda_1, \dots\, \lambda_k)$ and a positive integer $n$. The system $\TAZRP(\lambda,n)$ has $n$ sites (labelled $1,2,\dots,n$), and $k$ particles with types $\lambda_1,\dots, \lambda_k$. Each site may be empty or may contain one or more particles. Particles of the same type are indistinguishable.

The system evolves as a continuous-time Markov chain. Any jump of the system consists 
of a single particle jumping from site $j$ to site $j-1$, for some $j\in\{1,\dots,n\}$ 
(sites are considered cyclically mod $n$, so that a particle jumping out of site $1$
enters site $n$). The rates are governed by a global parameter $t$ and site-dependent parameters $x_1, \dots, x_n$.

For each $j\in\{1,\dots,n\}$ and each $a\geq 1$, a bell of level $a$ rings at site $j$ at rate
$x_j^{-1} t^{a-1}$. 
When such a bell rings: if site $j$ contains at least $a$ particles, 
then the $a$'th highest-numbered of them jumps to site $j-1$. If $j$ contains fewer
than $a$ particles, nothing changes.

In \cite{AMM2} we show that 
the stationary probabilities of the model are rational functions of 
$t$ and $x_1, \dots, x_n$ with non-negative integer coefficients, with 
common denominator
$\widetilde{H}_{\lambda}(x_1,\ldots,x_n;1,t)$. 
In fact, there is a function $f$ from the set $\PQT(\lambda, n)$
to the set of TAZRP configurations, 
such that the stationary probability $P(\eta)$ of a configuration $\eta$ is given 
by
\begin{equation}\label{TAZRPprobabilities}
P(\eta)=
\frac{1}{
\widetilde{H}_{\lambda}(x_1,\ldots,x_n;1,t)
}
{
\sum_{
\substack{
\sigma\in\PQT(\lambda,n): 
\\
f(\sigma)=\eta
}
}
x^\sigma t^{\quinv(\sigma)}
}.
\end{equation}
In this sense we can say that the functions $\widetilde{H}_{\lambda}$ 
act as  ``partition functions'' for the TAZRP.
The value $f(\sigma)$ depends on the tableau $\sigma$ only through its bottom row. 

The proof of property (\ref{TAZRPprobabilities}) in \cite{AMM2} is done by constructing a Markov chain on the space $\PQT(\lambda, n)$ in which the stationary probability of a tableau $\sigma$ is proportional to $x^\sigma t^{\quinv(\sigma)}$, and which projects via the function $f$ to $\TAZRP(\lambda,n)$.

\section{Background and preliminaries}
\label{sec:background}

The proofs of several of our results will rely on the techniques used in \cite{HHL05}, which we present in \cref{sec:axioms,sec:LLT,sec:superization}. The final subsection \cref{sec:qseries} is needed for the proofs in \cref{sec:bijwords}.

\subsection{The axioms uniquely characterizing $\widetilde{H}_{\lambda}$}\label{sec:axioms}

Let $\Lambda \equiv \Lambda(q,t)$ be the algebra of symmetric functions whose coefficients are rational functions in $q$ and $t$.
Recall that $\{p_{\mu}\}$ is the power sum basis for $\Lambda$, and $\omega$ is the standard involution on $\Lambda$. For a formal power series $A$ in indeterminates $a_1,a_2,\cdots$, in our case with coefficients in $\mathbb{Q}(q,t)$, $p_k[A]$ is the formal substitution of $a_i^k$ for each indeterminate $a_i$. Then for an arbitrary $f\in\Lambda$, the \emph{plethysm} $f[A]$ is defined by expressing $f$ in the power sum basis and substituting $p_k[A]$ for each $p_k$ in the expansion. By convention, we define the \emph{plethystic alphabets} $X=x_1+x_2+\cdots$ and $Y=y_1+y_2+\cdots$, so that $f[X]=f(X)$, $f[-X]=(-1)^d\omega (f(X))$ if $f$ is homogeneous of degree $d$, $f[X+Y]=f(X,Y)$, where $f(X,Y)$ represents the concatenation of the alphabets $X$ and $Y$, and $f[X(1-q)]$ is the image of $f$ under the algebra homomorphism mapping $p_k(X)$ to $(1-q^k)p_k(X)$. See \cite[\S 2]{Hai99} for a complete description.
 
The modified Macdonald polynomials are the basis of $\Lambda$ with coefficients in $\mathbb{Q}(q,t)$, characterized by the following triangularity and normalization axioms (derived from Macdonald's original triangularity and orthogonality axioms \cite{Hai99}), and symmetric functions satisfying the axioms are unique, if they exist:  see \cite[Proposition 2.6]{Hai99} and \cite[Section 6.1]{Haiman02}.

\begin{align}
\widetilde{H}_{\lambda}[X(1-t);q,t]& = \sum_{\mu \geq \lambda'} a_{\mu\lambda}(q,t)s_{\mu}(X), \label{eq:A1}\\
\widetilde{H}_{\lambda}[X(1-q);q,t]& = \sum_{\mu \geq \lambda} b_{\mu\lambda}(q,t)s_{\mu}(X), \label{eq:A2}\\
\left\langle \widetilde{H}_{\lambda}(X;q,t),s_{(n)}(X) \right\rangle  & =1 \label{eq:A3}
\end{align}
for some coefficients $a_{\mu\lambda}(q,t), b_{\mu\lambda}(q,t) \in \mathbb{Q}(q,t)$. 

We first use the following facts to rewrite the axioms in a more convenient form. 
\begin{enumerate}
\item $s_{\lambda}$ and $m_{\lambda}$ are lower triangular with respect to each other: $s_{\lambda}\in\mathbb{Z}\{m_{\mu}\,:\,\mu \leq \lambda\}$ and $m_{\lambda}\in\mathbb{Z}\{s_{\mu}\,:\,\mu \leq \lambda\}$.

\item We can write $f[X(q-1)] = f(qX,-X)$, where $f(X,Y)$ represents the concatenation of the plethystic alphabets $X$ and $Y$.

\item For a symmetric function $f(X)$ that is homogeneous of degree $d$ and a plethystic alphabet $Y$, $f(-Y) = (-1)^d(\omega f)(Y)$.

\item The partial ordering on partitions is reversed by transposing: $\lambda \leq \mu \longleftrightarrow \mu' \leq \lambda'$, where $\leq$ is the dominance order.
\end{enumerate}

Thus we rewrite the two triangularity axioms, which, along with \eqref{eq:A3} uniquely characterize $\widetilde{H}_{\lambda}$, in terms of the monomial basis: 
\begin{align}
\widetilde{H}_{\lambda}[X(t-1);q,t]& = \sum_{\mu \leq \lambda} d_{\mu\lambda}(q,t)m_{\mu}(X) \label{A1m}\\
\widetilde{H}_{\lambda}[X(q-1);q,t]& = \sum_{\mu \leq \lambda'} c_{\mu\lambda}(q,t)m_{\mu}(X), \label{A2m}
\end{align}
for some coefficients $c_{\mu\lambda}(q,t), d_{\mu\lambda}(q,t) \in \mathbb{Q}(q,t)$.

\subsection{LLT polynomials}
\label{sec:LLT}

LLT polynomials are a well-known family of symmetric polynomials discovered by Lascoux, Leclerc, and Thibon \cite{LLT}. We provide the combinatorial definition of LLT polynomials, which was introduced in \cite{LLT2}. 

Recall that the Young diagram of a partition $\lambda$ is a left-justified array of cells such that the $i$'th row contains $\lambda_i$ cells. We will number our rows from bottom to top, the so-called {\em French convention}. 

We will identify a partition with its Young diagram. Let $\lambda$ and $\mu$ be partitions with $\mu_j \leq \lambda_j$ for all $j$, i.e.~$\mu \subseteq \lambda$ as Young diagrams. Then the {\em skew diagram} or {\em skew shape}, denoted $\nu=\lambda / \mu$, is the subset of 
$\mathbb{Z}_+ \times \mathbb{Z}_+$ consisting of the cells in $\lambda/ \mu$. We imagine diagonals running through the cells of $\nu$ from southwest to northeast, 
and we define the \emph{diagonal} of a cell $u=(i,j)$ (in row $i$ and column $j$) to be the integer $d(u)=i-j+1$.

A \emph{semistandard Young tableau} (SSYT) of skew shape $\nu$ is a filling of the diagram $\nu$ with positive integers, denoted by the function $\rho:\nu \rightarrow \mathbb{Z}_+$, which is weakly increasing on each row of $\nu$ (from left to right) and strictly decreasing on each column (from bottom to top). We denote the set of fillings of $\nu$ by $\SSYT(\nu)$. For a filling $\rho\in\SSYT(\nu)$, $\rho(u)$ denotes the entry in cell $u$ of $\nu$.

Let $\boldsymbol{\nu}=(\nu^{(1)},\ldots,\nu^{(k)})$ be a tuple of skew diagrams, and let 
\[
\SSYT(\boldsymbol{\nu})=\SSYT(\nu^{(1)})\times\cdots\times\SSYT(\nu^{(k)}).
\] 
We note that although the representation $\nu=\lambda/\mu$ is not unique for a given diagram $\nu$, we need consider only the \emph{relative positions} of the $\nu^{(i)}$'s with respect to each other. Thus for each $i$, we fix $\nu^{(i)}$ to be such that the site $u_i$ in row 1, column 1 (whether or not there is a cell at that location) is positioned at the origin so that $d(u_i)=1$. 

\begin{defn}\label{def:LLTinversions}
The set of \emph{inversions} on $\boldsymbol{\rho}=(\rho^{(1)},\ldots,\rho^{(k)})\in\SSYT(\boldsymbol{\nu})$ is defined as follows. Let $u$, $v$ be cells in $\nu^{(i)}$ and $\nu^{(j)}$ respectively.
The cells $u$ and $v$ form an inversion if $\rho^{(i)}(u)>\rho^{(j)}(v)$ and either
\begin{enumerate}[label=\roman*.]
\item $i<j$ and $d(u)=d(v)$, or
\item $i>j$ and $d(u)=d(v)+1$.
\end{enumerate}
\end{defn}

We denote the number of inversions of $\boldsymbol{\rho}$ by $\inv(\boldsymbol{\rho})$, and 
the monomial in $x$ corresponding to the content of $\boldsymbol{\rho}$ by
$x^{\boldsymbol{\rho}} = \prod_i \prod_{u\in \nu^{(i)}} x_{\rho^{(i)}(u)}$. 

\begin{example}
\cref{fig:LLT1} shows a semistandard filling $\boldsymbol{\rho}=(\rho^{(1)},\rho^{(2)},\rho^{(3)})$ of the tuple of skew diagrams $\boldsymbol{\nu}$ where $\nu^{(1)}=(1,1)/\emptyset$, $\nu^{(2)}=(1,1)/\emptyset$, and $\nu^{(3)}=(2,2,2)/(2,1)$. There are five inversions in $\boldsymbol{\rho}$: the 3 in diagonal 3 forms an inversion with the 2 in diagonal 2, the 3 and the 2 in diagonal 2 form an inversion, the 2 in diagonal 2 forms an inversion with the bottommost 1 in diagonal 1, and the 4 in diagonal 2 forms an inversion with both 1's in diagonal 1. Thus $\inv(\boldsymbol{\rho})=5$ and $x^{\boldsymbol{\rho}}=x_1^2x_2^2x_3^2x_4$.
\end{example}

\begin{figure}[H]
\centering
\begin{tikzpicture}[node distance=2cm]

\def \d {0.4};
\def \h{0};
\draw[dashed,gray] (\h+3.1+\d,-1.5+\d)--(\h+6.1,1.5);
\node at (\h+3+\d,-1.7+\d) {$3$};
\draw[dashed,gray] (\h+3.1+2*\d,-1.5+\d)--(\h+6.1+\d,1.5);
\node at (\h+3+2*\d,-1.7+\d) {$2$};
\draw[dashed,gray] (\h+3.1+3*\d,-1.5+\d)--(\h+6.1+2*\d,1.5);
\node at (\h+3+3*\d,-1.7+\d) {$1$};
\node at (\h+2.5,-1.7+\d) {$d(u)$};

\node at (\h+4.2+\d,-1+\d) {$\tableau{3\\1}$};
\node at (\h+4+3*\d,-1.+2.5*\d) {$\tableau{2\\1}$};
\node at (\h+4+5*\d,-.8+4.5*\d) {$\tableau{3&4\\&2}$};
\end{tikzpicture}
\caption{For $\boldsymbol{\nu}=((1,1),(1,1),(2,2,2)/(2,1))$, we show a filling $\boldsymbol{\rho}\in\SSYT(\boldsymbol{\nu})$, with $\inv(\boldsymbol{\rho})=5$ and $x^{\boldsymbol{\rho}}=x_1^2x_2^2x_3^2x_4$.}
\label{fig:LLT1}
\end{figure}

\begin{defn}\label{def:LLT}
Let $\boldsymbol{\nu}$ be a tuple of skew diagrams. The LLT polynomial indexed by $\boldsymbol{\nu}$ is
\[
G_{\boldsymbol{\nu}}(X;t) = \sum_{\boldsymbol{\rho}\in\SSYT(\boldsymbol{\nu})} t^{\inv(\boldsymbol{\rho})}x^{\boldsymbol{\rho}}.
\]

\end{defn}
\begin{theorem}[\cite{LLT,LLT2}]\label{thm:LLT} 
The polynomial $G_{\boldsymbol{\nu}}(X;t)$ is symmetric in the variables $x$.
\end{theorem}
\begin{remark} Our \cref{thm:symmetry} relies on \cref{thm:LLT}. The original proof for \cref{thm:LLT} in \cite{LLT} uses a construction in the representation theory of affine Hecke algebras, and a purely combinatorial proof of the symmetry is given in \cite[Section 10: Appendix]{HHL05}.
\end{remark}

\begin{defn}\label{def:ribbon}
Define a \emph{ribbon} to be a connected skew diagram with no $2\times2$ squares. For a ribbon $\nu$ with $k=|\nu|$, we label its boxes from northwest to southeast by $1,\ldots,k$, and define its
{\em descent set}, denoted $\Des(\nu)$, to be the set of labels in $\{1,\ldots,k-1\}$ corresponding to boxes that have a box below in the same column.

Let $w$ be a word with entries in $\mathbb{Z}_+$. Denote the set of locations of descents in $w$ by $\Des(w) = \{i\,:\,w_i>w_{i+1}\}$. Let $\nu$ be the ribbon with the same descent set as $w$, i.e. $\Des(\nu)=\Des(w)$. We define a ribbon corresponding to $w$ to be a filling of the cells of $\nu$ with the entry $w_i$ in cell $i$ for $i=1,\ldots,k$. We call $\ribbon(w)$ the ribbon corresponding to $w$. 
\end{defn}

It is immediate from the definition that $\ribbon(w)$ is a SSYT.

\begin{example}
Consider the word $w=(4,3,3,4,5,3,2,1,2,2)$, which has descent set $\Des(w)= \{1,5,6,7\}$. The ribbon $\nu=(6,4,4,4,1)/(3,3,3)$ is the unique ribbon with the same descent set: $\Des(\nu)=\Des(w)$. Below we show $\nu$ with its cells labeled from northwest to southeast, and the corresponding $\SSYT$ $\ribbon(w)$ of shape $\nu$ with its boxes filled by the entries of $w$.

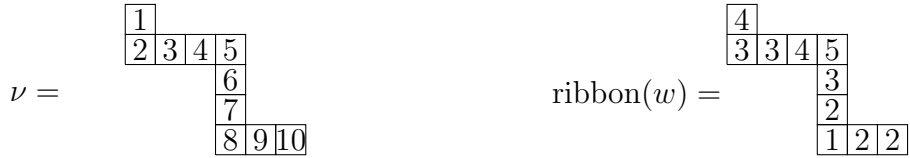
\begin{center}
\begin{tikzpicture}[scale=0.4]
\node at (-4,-3) {$\nu=$};
\cell10{$1$}\cell20{$2$} \cell21{$3$} \cell22{$4$} \cell23{$5$}  \cell33{$6$} \cell43{$7$} \cell53{$8$}  \cell54{$9$}  \cell55{$10$}

\node at (16,-3) {$\ribbon(w)=$};
\cell1{20}{$4$}\cell2{20}{$3$} \cell2{21}{$3$} \cell2{22}{$4$} \cell2{23}{$5$}  \cell3{23}{$3$} \cell4{23}{$2$} \cell5{23}{$1$}  \cell5{24}{$2$}  \cell5{25}{$2$}
\end{tikzpicture}
\end{center}

\end{example}

\subsection{Super-alphabets and quasisymmetric function expansion}\label{sec:superization}

With $\PQT$, $\quinv$, and $\maj$ as defined in \cref{sec:main}, we define:
\begin{equation}\label{eq:C}
C_{\lambda}(X;q,t) = \sum_{\sigma\in \PQT(\lambda)} q^{\maj(\sigma)} t^{\quinv(\sigma)}x^{\sigma}.
\end{equation} 

After showing $C_{\lambda}(X;q,t)$ is symmetric in the variables $x_i$ in \cref{sec:sym}, we will show it satisfies \eqref{A1m} and \eqref{A2m}. We do this by modifying the proof of the HHL formula in \cite[Section 4]{HHL05} for the setting of $\PQT(\lambda)$, where we will consider the \emph{superization} of $C_{\lambda}$. In this subsection, we adapt the well-known properties of superization to this new setting.

Let $n\geq 0$ and fix a subset $D\subseteq \{1,\ldots,n-1\}$. Define \emph{Gessel's quasisymmetric function} $Q_{n,D}(X)$ in the variables $X=x_1,x_2,\ldots$ by 
\[
Q_{n,D}(X) = \sum_{\substack{a_1\leq a_2\leq \cdots \leq a_n\\a_i=a_{i+1} \implies i\not\in D}} x_{a_1}x_{a_2}\cdots x_{a_n},
\]
where the indices are $a_i\in\mathbb{Z}_+$. Define the ``super-alphabet''
\[
\mathcal{A}=\mathbb{Z}_+ \cup \mathbb{Z}_- = \{\bar{1},1,\bar{2},2,\ldots\}
\]
consisting of positive and ``negative'' letters $i,\bar{i}$ of our original alphabet. One can consider any total ordering on $\mathcal{A} \cup \{0\}$; in our proofs we will use the two total orderings
\begin{align*}
(\mathcal{A}\cup\{0\},<_1)& = \{0<1<\bar{1}<2<\bar{2}<\cdots\},\\
(\mathcal{A}\cup\{0\},<_2)& = \{0<1<2<3<\cdots<\bar{3}<\bar{2}<\bar{1}\}.
\end{align*}
For any fixed total ordering $<$ on $\mathcal{A}$, we define the more general ``super'' quasiymmetric function $\widetilde{Q}_{n,D}(X,Y)$ in the variables $X=x_1,x_2,\ldots$ and $Y=y_1,y_2,\ldots$ by
\begin{equation}
\widetilde{Q}_{n,D}(X,Y) = \sum_{\substack{a_1\leq a_2\leq \cdots \leq a_n\\a_i=a_{i+1}\in\mathbb{Z}_+\implies i\not\in D\\a_i=a_{i+1}\in\mathbb{Z}_-\implies i\in D}} z_{a_1}z_{a_2}\cdots z_{a_n},
\end{equation}
where the indices are $a_i \in \mathcal{A}$, and $z_{i}=x_{i}$ when $i\in\mathbb{Z}_+$ and $z_{i}=y_{i}$ when $i\in\mathbb{Z}_-$.

For $a,b \in \mathcal{A}\cup\{0\}$ and any given total ordering $<$, we will use the notation $I(a,b)$: 
\[
I(a,b) = \begin{cases} 1, & a>b \mbox{ or } a=b \in \mathbb{Z}_-,\\
0, & a<b \mbox{ or } a=b \in \mathbb{Z}_+.
\end{cases}
\]

To avoid confusion, we will use the terminology $I_1$ ({\it resp.}~$I_2$) whenever we use the ordering $<_1$ ({\it resp.}~$<_2$).
For example, 
\begin{align*}
I_1(2,3)=I_1(2,\bar 2)=I_1(2, \bar 3)=I_1(2,2)=0, & \quad  
I_1(3,\bar 2)=I_1(\bar 3,\bar 3)=I_1(\bar 3,1)=1, \\ 
I_2(1,2)=I_2(2,\bar 2)=I_2(\bar 2,\bar 1)=I_2(2,2)=0, & \quad 
I_2(\bar 2,4)=I_2(\bar 2,\bar 2)=I_2(\bar 1,\bar 3)=1.
\end{align*}
 
\begin{defn}
The \emph{superization} of a symmetric function $f(X)$ is $\widetilde{f}(X,Y) = \omega_Y f[X+Y]$, where $\omega_Y$ acts on $f(x,y)$ considered as a symmetric function of the $Y$ variables.
\end{defn}

\begin{defn}\label{def:super}
Given a super alphabet $\mathcal{A} = \mathbb{Z}_+ \cup \mathbb{Z}_-$ and a fixed total ordering $<$, a \emph{super filling} of a diagram $\dg(\lambda)$ is a function $\sigma: \dg(\lambda) \rightarrow \mathcal{A}$, with the following extension of the
definitions of the maj and quinv statistics from \cref{def:maj} and \cref{def:quinv}. Denote the set of super fillings of $\dg(\lambda)$ by $\widetilde{\PQT}(\lambda)$.
\begin{itemize}[leftmargin=10pt]
\item $\maj$: If $y = \South(x)$ in $\dg(\lambda)$, then $(x,y) \in\Des(\sigma)$ if $I(\sigma(x),\sigma(y))=1$. The $\maj$ statistic is defined as before.

\item $\quinv$: If three cells $x,y,z$ with entries $\sigma(x)=a,\sigma(y)=b,\sigma(z)=c$ form a triple in the configuration 
\[
\begin{tikzpicture}[scale=.5]
\cellL0{.5}{$a$}{$x$}\cellL1{.5}{$b$}{$y$}\cellL1{3.5}{$c$}{$z$}
\node at (1.3,-.5) {$\cdots$};
\end{tikzpicture}
\] 
where $z$ is the cell to the right of $y$ in the same row, and $y=\South(x)$ if $x$ exists, then the triple is a $\quinv$ triple if and only if exactly \emph{one} of the following is true:
\[
\{I(a,b)=1,\ I(c,b)=0,\ I(a,c)=0\}.
\] 
It is \emph{not a $\quinv$ triple} if and only if exactly \emph{two} of the conditions above are true.\footnote{The reader may check that it is impossible for all or none of the conditions to be true.} 
As before, we write $\quinv(\sigma)$ as the number of $\quinv$ triples in $\sigma$.
\end{itemize}
We write $|\sigma|$ to denote the regular filling with the positive alphabet, such that $|\sigma|(u)= |\sigma(u)|$ for each $u\in \dg(\lambda)$. 
\end{defn}

It is immediate that when $\sigma=|\sigma|$, the above definitions reduce to those of the statistics of a regular filling as given in Section \ref{sec:definitions}. Moreover, note that the definition of $\quinv$ given above still holds for a degenerate triple, in which the cell $x$ does not exist: as per our convention, in that case $a=0$, so $I(a,b)=I(a,c)=0$, and hence the triple is a $\quinv$ triple if and only if $I(c,b)=1$.  

\begin{example}
We give some examples of $\quinv$ triples in super fillings when there are repeated entries, noting that $I(a,a)=0$ and $I(\bar{a},\bar{a})=1$ for any fixed ordering $<$. The following are $\quinv$ triples: 
\[
\qtrip{$\overline{1}$}{$\overline{1}$}{$\overline{1}$}, \ 
\qtrip{$2$}{$\overline{1}$}{$\overline{1}$}, \ 
\qtrip{$1$}{$1$}{$\overline{2}$}, \ 
\qtrip{$\overline{2}$}{$1$}{$\overline{2}$},
\]
whereas the following are not $\quinv$ triples:
\[
\qtrip{$\overline{1}$}{$\overline{1}$}{$2$}, \ \qtrip{$\overline{1}$}{$\overline{1}$}{$\overline{2}$}.
\]
Moreover, \begin{tikzpicture}[scale=0.5]\cell10{$\overline{2}$}\cell20{$\overline{2}$}\end{tikzpicture} is a descent, while (as before) \begin{tikzpicture}[scale=0.5]\cell10{$2$}\cell20{$2$} \end{tikzpicture} is not.
Since $I(a,a)=0$ and $I(\bar{a},\bar{a})=1$ for any ordering $<$, this example is independent of the ordering.
\end{example}

To obtain a tableaux formula for the generating function of the superization $\widetilde{C}(X,Y;q,t)$, we present a standard construction following an analogous argument in \cite{HHL05} that makes use of standard fillings to give a quasisymmetric expansion of our formulas. The following proposition states a well-known property of the superization of a symmetric function that holds for any total ordering on the super-alphabet $\mathcal{A}$. 

\begin{proposition}[{\cite[Corollary 2.4.3]{LLT2}}]\label{prop:superquasi}
Let $f(x)$ be a symmetric function homogeneous of degree $n$, written as a sum over quasisymmetric functions as
\[
f(z) = \sum_{D\subseteq \{1,\ldots,n-1\}} c_D Q_{n,D}(z).
\]
The superization of $f$ is given by
\[
\widetilde{f}(x,y) = \sum_{D\subseteq \{1,\ldots,n-1\}} c_D \widetilde{Q}_{n,D}(x,y).
\]
\end{proposition}

We call a filling $\pi$ of $\dg(\lambda)$ a \emph{standard filling} if each element in $\{1,\ldots,n\}$ occurs exactly once in $\pi$, where $n=|\lambda|$. A standard filling is represented by the bijection $\pi\ :\ \dg(\lambda) \rightarrow \{1,\ldots,n\}$.
The \emph{standardization} of a super filling $\sigma$ is the unique standard filling $\pi$ such that $\sigma\circ\pi^{-1}$ is weakly increasing, and for each $x\in\mathcal{A}$, the restriction of $\pi$ to $\sigma^{-1}(\{x\})$ is increasing with respect to the reading order if $x\in\mathbb{Z}_+$ and decreasing if $x\in\mathbb{Z}_-$. It is straightforward to check that if $\pi$ is the standardization of $\sigma$, then for each $u,v\in\dg(\lambda)$, $I(\sigma(u),\sigma(v))=I(\pi(u),\pi(v))$, and so the statistics $\maj$ and $\quinv$ are preserved under standardization: $\maj(\sigma)=\maj(\pi)$ and $\quinv(\sigma)=\quinv(\pi)$. 
Note that both the standardization and the function $I$ depend on the choice of ordering; see \cref{example:std}.

\begin{example}\label{example:std}
The standardization of $\sigma$ with respect to the ordering $<_1$ is $\pi_1$, and the standardization with respect to the ordering $<_2$ is $\pi_2$.
\begin{center}
\begin{tikzpicture}[scale=0.5]
\cell002
\cell10{$\bar1$}\cell11{$\bar2$}\cell121
\cell203\cell212\cell22{$\bar1$}\cell231\cell242
\node at (-2,-0.5) {$\sigma=$};
\end{tikzpicture}
\begin{tikzpicture}[scale=0.5]
\cell005
\cell10{4}\cell11{8}\cell121
\cell209\cell217\cell22{3}\cell232\cell246
\node at (-2.5,-0.5) {$,\quad\pi_1=$};
\end{tikzpicture}
\begin{tikzpicture}[scale=0.5]
\cell003
\cell10{9}\cell11{7}\cell121
\cell206\cell215\cell22{8}\cell232\cell244
\node at (-2.5,-0.5) {$,\quad\pi_2=$};
\end{tikzpicture}
\end{center}
\end{example}

Now consider the reading word $\rw(\pi)$ of a standard filling $\pi$ of $\dg(\lambda)$, which is defined to be the sequence of entries obtained from the filling in reading order: this is a permutation of $\{1,\ldots,n\}$ where $n=|\lambda|$. We call $D(\pi)\subseteq\{1,\ldots,n-1\}$ the \emph{index} of $\pi$, defined by
\[
D(\pi)=\big\{i \in \{1,\ldots,n-1\}\ :\ i+1 \text{ precedes } i \text{ in } \rw(\pi)\big\}.
\]
Then $\pi$ is the standardization of $\sigma$ if and only if the weakly increasing function $a=\sigma\circ \pi^{-1}:\{1,\ldots,n\}\rightarrow \mathcal{A}$ satisfies:
\begin{itemize}
\item if $a(i)=a(i+1)\in\mathbb{Z}_+$, then $i\not\in D(\pi)$, and
\item if $a(i)=a(i+1)\in\mathbb{Z}_-$, then $i\in D(\pi)$.
\end{itemize} 
\begin{example}
In \cref{example:std}, $D(\pi_1)=\{3,4,7\}$ and $D(\pi_2)=\{2,6,8\}$. We check that $\sigma\circ\pi_1^{-1}=(1,1,\bar{1},\bar{1},2,2,2,\bar{2},3)$ implies $1,5,6 \not\in D(\pi_1)$ and $3\in D(\pi_1)$ which is indeed true, and similarly $\sigma\circ\pi_2^{-1}=(1,1,2,2,2,3,\bar{2},\bar{1},\bar{1})$ implies $1,3,4\not\in D(\pi_2)$ and $8\in D(\pi_2)$, also true. 
\end{example}

Thus we obtain, by \cref{thm:symmetry} and \cref{prop:superquasi}, the following proposition, corresponding to \cite[Proposition 4.3]{HHL05} with our statistic $\quinv$ replacing the ``$\inv$'' statistic in the latter.

\begin{proposition}[{\cite[Proposition 4.3]{HHL05}}]
Let $\lambda$ be a partition of $n$. The polynomial $C_{\lambda}(x;q,t)$ has the following quasisymmetric expansion as a sum over standard fillings $\pi$ of $\dg(\lambda)$:
\[
C_{\lambda}(x;q,t) = \sum_{\pi} q^{\maj(\pi)}t^{\quinv(\pi)}Q_{n,D(\pi)}(x).
\]
The superization of $C_{\lambda}(x;q,t)$ has the expansion
\[
\widetilde{C}_{\lambda}(x,y;q,t) = \sum_{\pi} q^{\maj(\pi)}t^{\quinv(\pi)}\widetilde{Q}_{n,D(\pi)}(x,y).
\]
This has the following formula in terms of super fillings:
\[
\widetilde{C}_{\lambda}(x,y;q,t) = \sum_{\sigma\in \dg(\lambda) \rightarrow \mathcal{A}} q^{\maj(\sigma)}t^{\quinv(\sigma)}z^{\sigma},
\] 
where $z_i = x_i$ if $i\in\mathbb{Z}_+$ and $z_i = y_i$ if $i\in\mathbb{Z}_-$, and statistics $\quinv$ and $\maj$ on super fillings $\sigma\in \dg(\lambda) \rightarrow \mathcal{A}$ given as in \cref{def:super}. 
\end{proposition}

Denote the set of super fillings $\{\sigma\in \dg(\lambda) \rightarrow \mathcal{A}\}$ by $\widetilde{\PQT}(\lambda)$. We use the identies $C_{\lambda}[X(t-1);q,t]=\widetilde{C}_{\lambda}(tX,-X;q,t)$ and $C_{\lambda}[X(q-1);q,t]=\widetilde{C}_{\lambda}(qX,-X;q,t)$ to obtain
\begin{align}
C_{\lambda}[X(t-1);q,t]& = \sum_{\sigma\in \widetilde{\PQT}(\lambda)} (-1)^{m(\sigma)} q^{\maj(\sigma)}t^{p(\sigma)+\quinv(\sigma)}x^{|\sigma|}, \label{C1}\\
C_{\lambda}[X(q-1);q,t]& = \sum_{\sigma\in \widetilde{\PQT}(\lambda)} (-1)^{m(\sigma)} q^{p(\sigma)+\maj(\sigma)}t^{\quinv(\sigma)}x^{|\sigma|} \label{C2}
\end{align}
where $p(\sigma) = |\{u\,:\,\sigma(u)\in\mathbb{Z}_+\}|$ and $m(\sigma) = |\{u\,:\,\sigma(u)\in\mathbb{Z}_-\}|$ are the numbers of positive and negative entries in the super filling $\sigma$, respectively. Note that these formulas are valid for any total ordering chosen on $\mathcal{A}$. 

\subsection{Basic facts from $q$-series}\label{sec:qseries}

We recall some standard facts about $q$-series and combinatorics of words. 
Recall that the {\em $q$-numbers} are given by 
\[
[n] \equiv [n]_q := \frac{1-q^n}{1-q} = 1 + q + \cdots + q^{n-1}, 
\]
for $n \in \mathbb{N}$. For us, $q$ will be a fixed formal variable,
and we will not specify it for notational convenience.
The {\em $q$-factorial} is then
\[
[n]! := [1] [2] \cdots [n],
\]
and the {\em $q$-binomial coefficient} is
\[
\qbinom nk := \frac{[n]!}{[k]! [n-k]!}.
\]
Although it is not obvious that this is a polynomial, this can be seen
from the initial conditions $\qbinom nn = \qbinom n0 = 1$
and the generalized Pascal triangle recurrence
\begin{equation}
\label{qbinom-recur}
\qbinom {n+1}{m+1} = \qbinom nm + q^{m+1} \qbinom n{m+1}.
\end{equation}
The  {\em $q$-multinomial coefficient} is defined similarly. Suppose 
$\alpha = (\alpha_1,\dots, \allowbreak \alpha_n)$ is a tuple of nonnegative integers.
Recall that $|\alpha| = \alpha_1 + \cdots + \alpha_n$. Then
\[
\qbinom{|\alpha|}{\alpha_1,\dots,\alpha_n} := \frac{[|\alpha|]!}{[\alpha_1]! \cdots [\alpha_n]!}.
\]

For $\alpha = (\alpha_1,\dots, \alpha_n)$ a tuple of positive integers, let $W_\alpha$
be the set of words in the alphabet $[n]$ with $\alpha_i$ occurences of the letter $i$,
$1 \leq i \leq n$. For a word $w\in W_{\alpha}$, define $\coinv(w) = |\{(i,j)\ :\ i<j,\ w_i<w_j\}|$ to be the number of coinversions of $w$. The following result is classical.

\begin{proposition}[{\cite[Proposition 1.7.1]{stanley-ec1}}]
\label{prop:coinv-gf}
The $\coinv$ generating function of $W_\alpha$ is
\[
\sum_{w \in W_\alpha} q^{\coinv(w)} = \qbinom{|\alpha|}{\alpha_1,\dots,\alpha_n}.
\]
\end{proposition}

Strictly speaking, the result above is usually stated for the $\inv$ generating function for the number of inversions, but there is an easy bijection showing that the same result holds for the coinversion generating function as well.
We now list a few of the standard $q$-series identities that we will need in our proofs. The first is the 
well-known $q$-binomial theorem.

\begin{proposition}[{\cite[Equation (1.87)]{stanley-ec1}}]
\label{prop:qbinom}
\[
\sum_{i=0}^n x^i q^{\binom{i}{2}} \qbinom ni = \prod_{j=0}^{n-1} (1 + x q^j).
\]
\end{proposition}

The celebrated $q$-Chu-Vandermonde identity will prove very useful for us. 
We write it in the form more tractable for our purposes.

\begin{theorem}[{\cite[Equations (1.5.2) and (1.5.3)]{gasper-rahman-2004}}]
\label{thm:q-chuvan}
\[
\sum_{i=0}^k \qbinom m{k-i} \qbinom ni q^{i(m-k+i)} = \qbinom {m+n}k.
\]
\end{theorem}

The $q$-Chu-Vandermonde identity is also valid when $m,n$ are negative
integers. In that case, a formulation useful for us will be the following.

\begin{cor}
\label{cor:q-dualchu}
\[
\sum_{i=k}^{m-n+k} \qbinom {i}{k} \qbinom {m-i}{n-k} q^{i(n-k+1)} = \qbinom {m+1}{n+1} q^{k(n-k+1)}.
\]
\end{cor}

The last is a telescopic sum, which can be derived from the fundamental 
recurrence \eqref{qbinom-recur} for the $q$-binomial coefficients.

\begin{proposition}
\label{prop:q-telescope}
\[
\sum_{i=k}^n \qbinom {m-i}{n-i} q^{i(m-n+1)} = \qbinom {m-k+1}{n-k} q^{k(m-n+1)}.
\]
\end{proposition}

\section{Proof of \cref{thm:mainconj}: symmetry}
\label{sec:sym}

This section is devoted to proving the following theorem.

\begin{theorem}\label{thm:symmetry}
The polynomial $C_{\lambda}(X; q,t)$ is symmetric in the variables $x_i$.
\end{theorem}

We will prove \cref{thm:symmetry} by expanding $C_{\lambda}$ in terms of the LLT polynomials $G_{\boldsymbol{\nu}}(X;t)$ 
defined in \cref{sec:LLT}.

Let $\lambda$ be a partition with $k$ parts, let $\widehat{\dg}(\lambda)=\{(r,j) \in \dg(\lambda): r>1\}$ 
be the cells in $\dg(\lambda)$ not contained in the bottom row, and let $D\subseteq \widehat{\dg}(\lambda)$ be any subset. We define $\boldnu(\lambda,D)=(\nu^{(1)},\nu^{(2)},\ldots,\nu^{(k)})$ to be a tuple of $k$ ribbons (see \cref{def:ribbon}) such that 
\[
\Des(\nu^{(j)})=\{\lambda_j-r+1\ :\ (r,j)\in D\},
\]
where the ribbons are arranged from northeast to southwest, and such that the southeast-most cell of each ribbon is aligned on diagonal 1. In other words, the $j$'th ribbon has length $\lambda_j$ and its descent set $\Des(\nu^{(j)})$ corresponds to the restriction of $D$ to column $j$ of $\dg(\lambda)$ when read from top to bottom. See \cref{ex:LLT}.

\begin{example}\label{ex:LLT}
Let $\lambda=(3,2,2)$ with the subset of descents chosen to be\linebreak[4] $D=\{(2,3),(2,1)\}$, indicated by the shaded boxes in the figure below. Then $\Des(\nu^{(1)}) = \{1\}$, $\Des(\nu^{(2)})=\emptyset$, and $\Des(\nu^{(3)})=\{2\}$, and $\boldnu(\lambda,D)$ is the tuple of three ribbons shown below. 

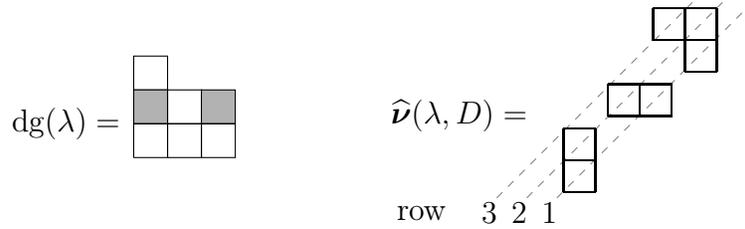
\begin{center}
\raisebox{1cm}{
\begin{tikzpicture}[scale=.45]
\node at (-3,-1) {$\dg(\lambda)=$};
\cell00{\ }
\graycell10{\ }\cell11{\ }\graycell12{\ }
\cell20{\ }\cell21{\ }\cell22{\ }
\end{tikzpicture}
}
\qquad\qquad
\begin{tikzpicture}
\def \d {0.4};
\def \h{0};
\def \v{-2};

\node at (\h+3,\v) {$\boldnu(\lambda,D)=$};
\draw[dashed,gray] (\h+3.1+\d,\v-1.5+\d)--(\h+6.1,\v+1.5);
\node at (\h+3+\d,\v-1.7+\d) {$3$};
\draw[dashed,gray] (\h+3.1+2*\d,\v-1.5+\d)--(\h+6.1+\d,\v+1.5);
\node at (\h+3+2*\d,\v-1.7+\d) {$2$};
\draw[dashed,gray] (\h+3.1+3*\d,\v-1.5+\d)--(\h+6.1+2*\d,\v+1.5);
\node at (\h+3+3*\d,\v-1.7+\d) {$1$};
\node at (\h+2.5,\v-1.7+\d) {row};

\node at (\h+4.2+\d,\v-1+\d) {$\tableau{\ \\\ }$};
\node at (\h+4.2+3*\d,\v-1.+.2+2.5*\d) {$\tableau{\ &\ }$};
\node at (\h+4+5*\d,\v-.8+4.5*\d) {$\tableau{\ &\ \\&\ }$};
\end{tikzpicture}
\end{center}
\end{example}

We now refine \eqref{eq:C} by splitting it into fillings with a given descent set.

\begin{defn}\label{def:invtwo}
Let $\lambda$ be a partition and $\sigma\in\PQT(\lambda)$. A pair of cells $u=(r,i)$ and $v=(r',j)$ is said to be \emph{attacking} if 
\begin{enumerate}[label=\roman*.]
\item $r=r'$ and $i>j$, or 
\item $r=r'+1$ and $i<j$,
\end{enumerate}
i.e. in the following configurations:
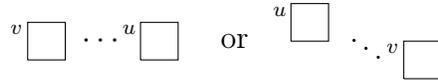
\begin{center}
\begin{tikzpicture}[scale=0.5]
\cellL{.5}0{}{$v$}\cellL{.5}3{}{$u$}
\node at (1,0) {$\cdots$};
\node at (4.5,0) {or};
\cellL07{}{$u$}\cellL1{10}{}{$v$}
\node at (8,0) {$\ddots$};
\end{tikzpicture}
\end{center}

If $\sigma(u)>\sigma(v)$ for a pair of attacking cells $u,v$, they form an \emph{attacking inversion}. If $\sigma(u)\neq\sigma(v)$ for every pair of attacking cells $u,v\in\dg(\lambda)$, we call $\sigma$ a \emph{non-attacking} filling.

Denote the number of attacking inversions in $\sigma\in\PQT(\lambda)$ by $\invtwo(\sigma)$. For a cell $u=(r,i)\in \dg(\lambda)$, we define $\armtwo(u)$ to be the number of cells $(r-1,j)\in\dg(\lambda)$ such that $j>i$, i.e. the shaded cells belong to $\armtwo$ of the cell labeled $a$ below. 
\begin{center}
\begin{tikzpicture}[scale=.45]
\cell00{\ }\cell01{$a$}\cell02{\ }
\cell10{\ }\cell11{\ }\graycell12{\ }\graycell13{\ }\graycell14{\ }
\cell20{\ }\cell21{\ }\cell22{\ }\cell23{\ }\cell24{\ }\cell25{\ }
\end{tikzpicture}
\end{center}
\end{defn}

\begin{remark} 
The $\,\widehat{\cdot}\,$ symbol will help differentiate $\widehat{\inv}$ and $\widehat{\arm}$ from the traditional definitions of the notions of inversions and arms that appear in the proofs of the corresponding HHL formulas  in \cite[Section 3]{HHL05}. Note that we have called $\invtwo$ ``attacking inversions'' because they correspond to inversions occurring in ``attacking cells'' with respect to the reading order, to parallel the terminology used in \cite{HHL05} when referring to their version of attacking cells and inversions. 
\end{remark}

We will be grouping tableaux by their descent sets, indexed by subsets $D\subseteq \widehat{\dg}(\lambda)$. For each subset $D\subseteq \widehat{\dg}(\lambda)$, define
\begin{equation}\label{eq:F_D}
F_{\lambda,D}(X;t) = \sum_{\substack{\sigma\in\PQT(\lambda)\\ \Des(\sigma)=D}}  t^{\invtwo(\sigma)} x^{\sigma}.
\end{equation}
Now, we rewrite $C_{\lambda}$ in terms of the $F_{\lambda,D}$'s.

\begin{defn}
Let $U\subseteq \dg(\lambda)$ be a subset of cells. Define
\[
\maj(U) = \sum_{u\in U} \leg(u)+1
\]
and
\[
\armtwo(U)=\sum_{u\in U} \armtwo(u).
\]
\end{defn}

\begin{lemma}
\[
C_{\lambda}(X; q,t) = \sum_{D\subseteq \widehat{\dg}(\lambda)} q^{\maj(D)} t^{-\armtwo(D)} F_{\lambda,D}(X;t).
\]
\end{lemma}
\begin{proof}
We write \eqref{eq:C} as a sum over descent sets:
\[
C_{\lambda}(X;q,t)=\sum_{D\subseteq \widehat{\dg}(\lambda)} q^{\maj(D)}\sum_{\substack{\sigma\in\PQT(\lambda)\\\Des(\sigma)=D}} t^{\quinv(\sigma)}x^{\sigma}.
\]
Now suppose $\sigma\in\PQT(\lambda)$ with $D:=\Des(\sigma)$. We will show that 
\begin{equation}\label{invtwo}
\quinv(\sigma)=\invtwo(\sigma)-\sum_{u\in D} \armtwo(u).
\end{equation}
Consider a triple $(x,y,z)\in\dg(\lambda)$ with respective contents $a:=\sigma(x)$, $b:=\sigma(y)$, $c:=\sigma(z)$ in the configuration  
\[
\begin{tikzpicture}[scale=.5]
\cellL0{.5}{$a$}{$x$}\cellL1{.5}{$b$}{$y$}\cellL1{3.5}{$c$}{$z$}
\node at (1.3,-.5) {$\cdots$};
\end{tikzpicture}
\ .
\] 
(The triple may be degenerate, in which case $a=0$.)

By comparing to \eqref{cQdef}, $(a,b,c)\in\cQ$ if and only if exactly one of the following is true:
\begin{equation}\label{quinv_conditions}
\left\{\begin{matrix}x \in D,\\(y,z) \text{ do not form an attacking inversion},\\ (z,x) \text{ do not form an attacking inversion}\end{matrix}\right\}.
\end{equation}

First suppose $x \not\in D$. If $(a,b,c)\in\cQ$, then exactly one of $(y,z)$ or $(z,x)$ contributes to $\invtwo(\sigma)$ according to \eqref{quinv_conditions}. If $(a,b,c)\not\in\cQ$, then neither $(y,z)$ or $(z,x)$ contributes to $\invtwo(\sigma)$. The contribution of the triple $(x,y,z)$ to the LHS and RHS of \eqref{invtwo} match in each case. 

Now suppose $x\in D$.  If $(a,b,c)\in\cQ$, then both $(y,z)$ and $(z,x)$ contribute to $\invtwo(\sigma)$, but since $z\in\armtwo(x)$, the total contribution to the RHS in \eqref{invtwo} is 1, which matches the contribution to the LHS. If $(a,b,c)\not\in\cQ$, then exactly one of $(y,z)$ or $(z,x)$ contributes to $\invtwo(\sigma)$, but since $z\in\armtwo(x)$, the total contribution to the RHS in \eqref{invtwo} is zero, again matching the contribution to the LHS.

Every pair $(u,v)$ where $u\in D$ and $v\in\armtwo(u)$, and every attacking pair that contributes to $\invtwo(\sigma)$, corresponds to some triple in $\sigma$. Thus every term in the LHS is accounted for, and \eqref{invtwo} follows.
\end{proof}

We will now describe a weight-preserving map $\widehat{\LLT}$ from $\PQT(\lambda)$ with a fixed descent set $D$ to fillings $\SSYT(\boldnu(\lambda,D))$.

\begin{defn}\label{def:LLTtwo}
For each descent set $D \subseteq \{(r,j) \in \widehat{\dg}(\lambda)\}$, define the map 
\[
\widehat{\LLT}: \{\sigma\in \PQT(\lambda):\Des(\sigma)=D\} \rightarrow \SSYT(\boldnu(\lambda,D))
\] 
as follows. Let $\sigma\in\PQT(\lambda)$, and let $c_1,\ldots,c_k$ be its columns from left to right, where each column contains the entries in $\sigma$ from top to bottom: $c_i = (\sigma(\lambda_i,i),\linebreak[0],\ldots,\sigma(2,i),\sigma(1,i))$. 
Set $\rho^{(i)} = \ribbon(c_i)$. 
Define $\boldsymbol{\rho}=(\rho^{(k)},\rho^{(k-1)},\ldots,\rho^{(1)})$ to be the tuple of ribbons corresponding to the columns in reverse order, such that the last boxes of each ribbon are on the same diagonal, and set $\widehat{\LLT}(\sigma) = \boldsymbol{\rho}$. See \cref{fig:LLT_T} for an example.
\end{defn}

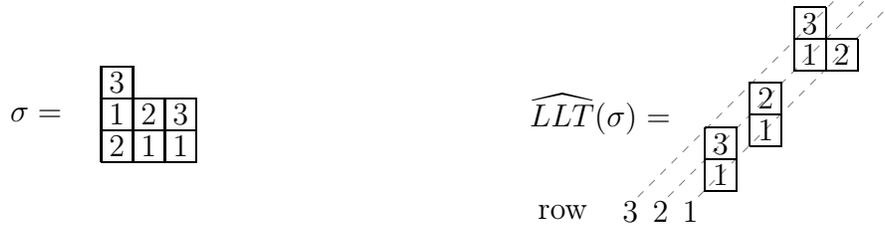
\begin{figure}[H]
\centering
\begin{tikzpicture}[node distance=2cm]
\node at (-1.5,0) {$\sigma=$};
\node at (0,0) {$\tableau{3\\1&2&3\\2&1&1}$};

\def \d {0.4};

\def \h{3};
\node at (\h+3,0) {$\widehat{\LLT}(\sigma)=$};
\draw[dashed,gray] (\h+3.1+\d,-1.5+\d)--(\h+6.1,1.5);
\node at (\h+3+\d,-1.7+\d) {$3$};
\draw[dashed,gray] (\h+3.1+2*\d,-1.5+\d)--(\h+6.1+\d,1.5);
\node at (\h+3+2*\d,-1.7+\d) {$2$};
\draw[dashed,gray] (\h+3.1+3*\d,-1.5+\d)--(\h+6.1+2*\d,1.5);
\node at (\h+3+3*\d,-1.7+\d) {$1$};
\node at (\h+2.5,-1.7+\d) {row};

\node at (\h+4.2+\d,-1+\d) {$\tableau{3\\1}$};
\node at (\h+4+3*\d,-1.+2.5*\d) {$\tableau{2\\1}$};
\node at (\h+4+5*\d,-.8+4.5*\d) {$\tableau{3\\1&2}$};
\end{tikzpicture}
\caption{The filling $\widehat{\LLT}(\sigma)$ corresponding to $\sigma\in\PQT(\lambda)$. The numbering of the diagonals is shown on the bottom. One can check that $\inv(\widehat{\LLT}(\sigma))=\invtwo(\sigma) = 5$. Moreover, the pairs of inversions in $\widehat{\LLT}(\sigma)$
correspond precisely to the attacking inversions in $\sigma$.
}
\label{fig:LLT_T}
\end{figure}

\begin{remark}
The careful reader may observe that the map $\widehat{\LLT}$ is identical to the corresponding map described in the proof of \cite[Proposition 3.4]{HHL05} if the columns of the filling of $\dg(\lambda)$ are taken in reverse order to be mapped to ribbons.
\end{remark}

From the definitions of LLT inversions in \cref{def:LLTinversions} and attacking inversions in fillings of diagrams in \cref{def:invtwo}, the following result is immediate. See \cref{fig:LLT_T} for an example.

\begin{lemma}\label{lem:LLTtwo}
Let $\sigma\in\PQT(\lambda)$ and $\boldsymbol{\rho}=\widehat{\LLT}(\sigma)$. Then
\[
\invtwo(\sigma) = \inv(\boldsymbol{\rho}).
\]
\end{lemma}

Thus we have defined a weight preserving bijection from $\PQT(\lambda)$ with descent set $D \subseteq \{(r,j) \in \widehat{\dg}(\lambda)\}$ to $\SSYT(\boldnu(\lambda,D))$. (The reverse map follows easily from reversing \cref{def:LLTtwo}.) From this we obtain our final key lemma.

\begin{lemma}\label{lem:FG}
Let $\lambda$ be a partition, and let $D \subseteq \{(r,j) \in \widehat{\dg}(\lambda)\}$. Then
\[
F_{\lambda,D}(X;t) = G_{\boldnu(\lambda,D)}(X;t).
\]
\end{lemma}

\cref{thm:symmetry} follows directly from \cref{thm:LLT} and \cref{lem:FG}.

\section{Proof that $C_{\lambda}(X;q,t)$ satisfies \eqref{A1m}}
\label{sec:A1m}

In this section we will use the ordering $<_1$ on $\mathcal{A}$ and construct a sign-reversing, weight-preserving involution $\Psi$ on super fillings $\widetilde{\PQT}(\lambda)$, which will cancel out all terms involving $x^{\mu}$ if $\mu \leq \lambda'$. 

We begin by defining the following map.

\begin{defn}\label{def:Phi_u} 
Let $u\in\dg(\lambda)$. For $\sigma\in\widetilde{\PQT}(\lambda)$, define the map $\Phi_u$ by
\[
\Phi_u(\sigma(w))=\begin{cases} \sigma(w),& w \neq u,\\
-\sigma(w),& w=u.
\end{cases}
\]
\end{defn} 

In other words, $\Phi_u$ is the map that negates the contents of the cell $u$. 

We first recall \cref{def:invtwo} for attacking cells in our setting. A pair of cells $u=(r,i)$ and $v=(r',j)$ with $i<j$ is attacking if $r=r'$ or $r=r'+1$, i.e. they are either in the same row, or the one to the right is one row below.

\vspace{0.1in}
\begin{defn}
\label{def:attackpair}
Let $\sigma \in\widetilde{\PQT(\lambda)}$. We define $\Psi(\sigma)$ as follows.
\begin{itemize}[leftmargin=10pt]
\item If there is no pair of attacking cells $u,v$ in $\sigma$ such that $|\sigma(u)|=|\sigma(v)|$, then set $\Psi(\sigma)=\sigma$.
\item Otherwise let $a$ be the smallest integer such that $|\sigma(x)|=|\sigma(y)|=a$ for some pair $x,y$ of attacking cells in $\sigma$. Let $v$ be the last cell in reading order among all such attacking pairs, and let $u$ be the last cell in reading order that attacks $v$ and such that $|\sigma(u)|=a$. Then $\Psi(\sigma)=\Phi_u(\sigma)$.
\end{itemize}
\end{defn}

We have defined $\Psi$ such that if $\sigma$ is not a fixed point, the map flips the sign of the entry at a designated cell $u$ that depends only on $|\sigma|$: thus $\Psi(\Psi(\sigma))=\sigma$. The following theorem is our main result in this section.

\begin{theorem}
\label{lem:A1FP}
\begin{align*}
C_{\lambda}[X(t-1);q,t]& = \sum_{\substack{\sigma\in\widetilde{\PQT}(\lambda)\\\sigma\,:\,\Psi(\sigma)=\sigma}} (-1)^{m(\sigma)} q^{\maj(\sigma)}t^{p(\sigma)+\quinv(\sigma)}x^{|\sigma|}\\
&=\sum_{\mu \leq \lambda} c_{\mu\lambda}(q,t)m_{\mu}.
\end{align*}
\end{theorem}

The proof of the theorem above will follow from the next two lemmas.

\begin{lemma}\label{lem:A1maj}
For any $\sigma\in\widetilde{\PQT}(\lambda)$, we have $\maj(\Psi(\sigma))=\maj(\sigma)$.
\end{lemma}

\begin{lemma}\label{lem:A1quinv}
If $\sigma$ is not a fixed point of $\Psi$ so that $\Psi(\sigma)=\Phi_u(\sigma)$ for some cell $u$, then $\Psi=\Phi_u$ changes the number of $\quinv$ triples by exactly one:
\[
\quinv(\Phi_u(\sigma))=\quinv(\sigma)+\begin{cases} -1,&\mbox{if}\ \sigma(u)\in\mathbb{Z}_+,\\ 1,&\mbox{if}\ \sigma(u)\in\mathbb{Z}_-.\\\end{cases}
\]
\end{lemma}

Since $u$ precedes $v$ in reading order in the last attacking pair with entries equal to $a$ in absolute value, the pair $u,v$ is in one of the two configurations:
\begin{center}
\begin{tikzpicture}[scale=0.5]
\cellL{.5}0{}{$v$}\cellL{.5}3{}{$u$}
\node at (1,0) {$\cdots$};
\node at (4.5,0) {or};
\cellL07{}{$u$}\cellL1{10}{}{$v$}
\node at (8,0) {$\ddots$};
\end{tikzpicture}
\end{center}

We first make some simple but key observations.

\begin{proposition}
\label{prop:attackpair}
Suppose $\sigma\in\widetilde{\PQT}(\lambda)$ is not a fixed point of $\Psi(\sigma)$, and let $u$ be such that $\Psi(\sigma)=\Phi_u(\sigma)$. Let $(u,v)$ be an attacking pair for $\sigma$ as given in \cref{def:attackpair}, and $|\sigma(u)|
= |\sigma(v)| = a$.

\begin{enumerate}
\item If there exists a cell $y$ below $u$, then $|\sigma(y)| \neq a$.

\item For any $b$ we have $I_1(b,a)=I_1(b,\bar{a})$.

\item For any $b$ such that $|b| \neq a$, we have $I_1(a,b)=I_1(\bar{a},b)$.
\end{enumerate}

\end{proposition}

\begin{proof}
For (1), note that $y$ forms an attacking pair together with $v$, and since $(y,v)$ comes after $(u,v)$ in the reading order, $|\sigma(y)| \neq a$. (2) and (3) are easily verified by going through the cases.
\end{proof}

\begin{proof}[Proof of \cref{lem:A1maj}]
We will show that $\Psi$ preserves the descent set of $\sigma$, which implies $\maj$ is also preserved:
\[
\Des(\Psi(\sigma)) = \Des(\sigma).
\]
If $\sigma$ is a fixed point of $\Psi$, there is nothing to prove, so suppose $\Psi(\sigma)=\Phi_u(\sigma)$ for some cell $u$. Since $\Phi_u$ is an involution, let us assume without loss of generality that $\sigma(u)=a\in\mathbb{Z}_+$. As $u$ is the only cell that changed in $\Phi_u(\sigma)$, the only cells that might have changed whether or not they are descents are $u$ and the cell directly above it. Let us consider the cells $x$ and $y$ directly above and below $u$, if they exist. Set $a=\sigma(u), b=\sigma(y), c=\sigma(x)$. Then we obtain the transition: 
\begin{equation}
\label{psi-action-maj}
\begin{tikzpicture}[scale=0.5]
\cellL01{$c$}{$x$}\cellL11{$a$}{$u$}\cellL21{$b$}{$y$}
\node at (3,-.5) {$\longrightarrow$};
\node at (3,.2) {$\Phi_u$};
\cellL06{$c$}{$x$}\cellL16{$\bar{a}$}{$u$}\cellL26{$b$}{$y$}
\end{tikzpicture}
\end{equation}
Since $|b|\neq a$ by \cref{prop:attackpair}(1), we have $u\in\Des(\sigma)$ if and only if $\bar{u}\in\Des(\Phi_u(\sigma))$. 
By \cref{prop:attackpair}(3), $c$ (if it exists) is a descent for $\sigma$ if and only if it is a descent for $\Phi_u(\sigma)$.
Thus  $x\in\Des(\sigma)$ if and only if $x\in\Des(\Phi_u(\sigma))$, thus concluding the proof.
\end{proof}

\begin{proof}[Proof of \cref{lem:A1quinv}]
Let us assume without loss of generality that $\sigma(u)=a\in\mathbb{Z}_+$. The only triples whose contribution to $\quinv(\sigma)$ is different from $\quinv(\Phi_u(\sigma))$ are ones that include $u$. Moreover, since $I_1(a,b)=I_1(\bar{a},b)$ for any $b$ such that $|b| \neq a$ by \cref{prop:attackpair}(3), the only triples we need to consider are ones that include $u$ plus another entry with absolute value equal to $a$. We examine the following three possible types of triples containing $u$, for some entries $b,c$ in cells $y,x$ respectively.
\begin{flushleft}
\begin{tikzpicture}[scale=0.5]
\node at (-4.5,0) {\emph{Case 1}:};
\cellL0{.5}{$a$}{$u$}\cellL1{.5}{$c$}{$x$}\cellL1{3}{$b$}{$y$}
\node at (1.3,-.5) {$\cdots$};
\node at (5,0) {$\longrightarrow$};
\node at (5,.7) {$\Phi_u$};
\cellL08{$\bar{a}$}{$u$}\cellL18{$c$}{$x$}\cellL1{10.5}{$b$}{$y$}
\node at (8.8,-.5) {$\cdots$};
\end{tikzpicture}
\end{flushleft}
Since $|c|\neq a$ by \cref{prop:attackpair}(1), suppose $|b|=a$. For both $b=a$ and $b=\bar{a}$, we have $I_1(a,b)=0$, $I_1(\bar{a},b)=1$, and also $I_1(b,c)=I_1(a,c)=I_1(\bar{a},c)$ by \cref{prop:attackpair}(3). Then, exactly two of the conditions
\[\{I_1(a,c)=1,\ I_1(b,c)=0,\  I_1(a,b)=0\}\]
are true in $\sigma$ and similarly, exactly one of the conditions
\[\{I_1(\bar{a},c)=1,\ I_1(b,c)=0,\ I_1(\bar{a},b)=0\}\]
is true in $\Phi_u(\sigma)$; hence the triple $(u,x,y)$ is a $\quinv$ triple in $\Phi_u(\sigma)$, but not in $\sigma$.
\begin{flushleft}
\begin{tikzpicture}[scale=0.5]
\node at (-4.5,0) {\emph{Case 2}:};
\cellL0{.5}{$b$}{$y$}\cellL1{.5}{$a$}{$u$}\cellL1{3}{$c$}{$x$}
\node at (1.3,-.5) {$\cdots$};
\node at (5,0) {$\longrightarrow$};
\node at (5,.7) {$\Phi_u$};
\cellL08{$b$}{$y$}\cellL18{$\bar{a}$}{$u$}\cellL1{10.5}{$c$}{$x$}
\node at (8.8,-.5) {$\cdots$};
\end{tikzpicture}
\end{flushleft}
For any $x,y$ (including when $y$ does not exist) we have $I_1(c,a)=I_1(c,\bar{a})$ and $I_1(b,a)=I_1(b,\bar{a})$ by \cref{prop:attackpair}(2), and so the triple $(y,u,x)$ is trivially a $\quinv$ triple in $\sigma$ if and only if it is one in $\Phi_u(\sigma)$.

\begin{flushleft}
\begin{tikzpicture}[scale=0.5]
\node at (-4.5,0) {\emph{Case 3}:};
\cellL0{.5}{$c$}{$x$}\cellL1{.5}{$b$}{$y$}\cellL1{3}{$a$}{$u$}
\node at (1.3,-.5) {$\cdots$};
\node at (5,0) {$\longrightarrow$};
\node at (5,.7) {$\Phi_u$};
\cellL08{$c$}{$x$}\cellL18{$b$}{$y$}\cellL1{10.5}{$\bar{a}$}{$u$}
\node at (8.8,-.5) {$\cdots$};
\end{tikzpicture}
\end{flushleft}
Again, $I_1(c,a)=I_1(c,\bar{a})$ for all $c$ by \cref{prop:attackpair}(2), and $I_1(a,b)=1-I_1(\bar{a},b)$ if and only if $|b|=a$. As in Case 1, if $|b|=a$, the triple $(x,y,u)$ is a $\quinv$ triple in $\Phi_u(\sigma)$, but not in $\sigma$.

Fix $y$ to be the cell with entry $b:=\sigma(y)$ in the figures above. Next we show that if Case 1 or Case 3 occurs for a triple with $|\sigma(y)|=a$, then necessarily $y=v$, i.e. $u$ and $y$ is precisely the last attacking pair in reading order whose entries have absolute value $a$. It is easy to check this by considering all the following possible configurations of $u,v,y$ below (the first row showing Case 1 and the second row showing Case 3), that if $v\neq y$, the last pair of attacking cells with absolute value $a$ must then include $y$, which is a contradiction.

\begin{center}
\begin{tikzpicture}[scale=0.5]
\cellL0{2.5}{}{$u$}\cellL1{5}{}{$y$}\cellL1{7.5}{}{$v$}
\node at (3.3,0){$\ddots$};
\node at (5.8,-.5) {$\cdots$};
\end{tikzpicture}\ ,
\quad
\begin{tikzpicture}[scale=0.5]
\cellL0{2.5}{}{$u$}\cellL1{5}{}{$v$}\cellL1{7.5}{}{$y$}
\node at (3.3,0){$\ddots$};
\node at (5.8,-.5) {$\cdots$};
\end{tikzpicture}\ ,
\quad
\begin{tikzpicture}[scale=0.5]
\cellL00{}{$v$}\cellL0{2.5}{}{$u$}\cellL1{5}{}{$y$}
\node at (.8,.5) {$\cdots$};
\node at (3.3,0){$\ddots$};
\end{tikzpicture}\ ,
\quad
\begin{tikzpicture}[scale=0.5]
\cellL00{}{$y$}\cellL0{2.5}{}{$v$}\cellL05{}{$u$}
\node at (.8,.5) {$\cdots$};
\node at (3.3,.5) {$\cdots$};
\end{tikzpicture}
\quad
\begin{tikzpicture}[scale=0.5]
\cellL00{}{$y$}\cellL0{2.5}{}{$u$}\cellL1{5}{}{$v$}
\node at (.8,.5) {$\cdots$};
\node at (3.3,0){$\ddots$};
\end{tikzpicture}\ ,
\quad
\begin{tikzpicture}[scale=0.5]
\cellL00{}{$v$}\cellL0{2.5}{}{$y$}\cellL05{}{$u$}
\node at (.8,.5) {$\cdots$};
\node at (3.3,.5) {$\cdots$};
\end{tikzpicture}
\end{center}

Therefore, every triple in $\sigma$ contributes to $\quinv(\sigma)$ if and only if it also contributes to $\quinv(\Phi_u(\sigma))$, with the exception of exactly one triple, namely the unique triple that contains both $u$ and $v$. Since we have assumed $\sigma(u)=a\in\mathbb{Z}_+$, this triple does not contribute to $\quinv(\sigma)$, but does contribute to $\quinv(\Phi_u(\sigma))$, so we obtain $\quinv(\Phi_u(\sigma))=\quinv(\sigma)+1$, giving us the desired expression.
\end{proof}

\begin{proof}[Proof of \cref{lem:A1FP}]
By \cref{lem:A1maj,lem:A1quinv}, if $\sigma$ is not a fixed point of $\Psi$, then $\Psi(\sigma)=\Phi_u(\sigma)$ for some cell $u$, so that $p(\Phi_u(\sigma)) = p(\sigma)+\begin{cases}-1, & \mbox{if}\ \sigma(u)\in\mathbb{Z}_+,\\1, & \mbox{if}\ \sigma(u)\in\mathbb{Z}_-.\end{cases}$ Thus 
\[
q^{\maj(\Psi(\sigma))}t^{p(\Psi(\sigma))+\quinv(\Psi(\sigma))} = q^{\maj(\sigma)}t^{p(\sigma)+\quinv(\sigma)},
\] 
and so the contribution to the RHS of \eqref{C1} of $\Psi(\sigma)$ differs by a negative sign from the contribution of $\sigma$, and so these terms cancel each other out. Hence we obtain the first equality.

For the second equality, we use the fact that the fixed points of $\Psi$ are precisely the non-attacking super fillings of $\dg(\lambda)$, which implies in particular that in each row, an entry appears at most once in absolute value. Suppose the monomial $x^{\alpha}$ appears as a term in the sum. Since the polynomial is symmetric, let us assume $\alpha=(\alpha_1,\alpha_2,\ldots)$ such that $\alpha_1\geq \alpha_2 \geq \cdots$. For each $j$, $\alpha_1+\cdots+\alpha_j$ is the number of entries in $\sigma$ with absolute value at most $j$. Recall that the rows of $\dg(\lambda)$ are of lengths $(\lambda_1',\lambda_2',\ldots)$ and columns are of lengths $(\lambda_1,\lambda_2,\ldots)$. Consequently, counting by rows, the number of entries with value at most $j$ counted by $\alpha_1+\cdots+\alpha_j$ cannot exceed $\sum_i \min(\lambda_i',j) = \lambda_1+\cdots+\lambda_j$, which is the condition that $\alpha \leq \lambda$. 
\end{proof}

\section{Proof that $C_{\lambda}(X;q,t)$ satisfies \eqref{A2m} in the non-degenerate case}
\label{sec:A2m}

In this section, we will use the ordering $<_2$ on $\mathcal{A}$ and describe sign-reversing, weight-preserving maps on super fillings $\widetilde{\PQT}$, which will cancel out all terms in \eqref{C2} involving $x^{\mu}$ if $\mu \leq \lambda$. We will follow the strategy of the proof of the corresponding result for HHL tableaux in \cite[Section 5.2]{HHL05}. However, our proof will deviate for a particular subset of fillings, which we describe below as \emph{$\Phi$-degenerate} fillings, and will treat separately in \cref{sec:Phi-degenerate}.

\begin{defn}
For $\sigma\in\widetilde{\PQT}(\lambda)$, let $a\in\mathbb{Z}_+$ be the smallest positive integer such that there exists a cell $(r,j)\in\dg(\lambda)$ with $|\sigma((r,j))|=a$ such that $r>a$, if such an $a$ exists. We call such an $a$ the \emph{distinguished label} of $\sigma$. 
If $a$ exists, we call the first cell in $\PQT$ reading order whose absolute value is $a$ the \emph{distinguished cell} of $\sigma$.
Then $\sigma$ falls into one of the following three categories:
\begin{enumerate}[label=(\alph*),leftmargin=20pt]
\item if no such $a$ exists, we say $\sigma$ is \emph{$\Phi$-trivial}.
\item if the distinguished cell does not belong to any degenerate triples (regardless of whether they are $\quinv$ triples), we say $\sigma$ is \emph{$\Phi$-nondegenerate}.
\item if the distinguished cell is part of any degenerate triples (regardless of whether they are $\quinv$ triples), we say $\sigma$ is \emph{$\Phi$-degenerate}. If $r$ is the row containing the distinguished cell, we call the set of cells in row $r$ that form degenerate triples the \emph{degenerate segment}.
\end{enumerate}
See \cref{fig:degenerate} for examples of all three.
\end{defn}

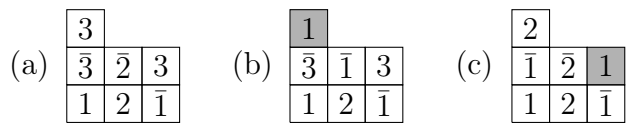
\begin{figure}[h]
\centering
\begin{tikzpicture}[scale=0.5]
\cell003
\cell10{$\bar3$}\cell11{$\bar2$}\cell123
\cell201\cell212\cell22{$\bar1$}
\node at (-2,-0.5) {(a)};
\end{tikzpicture}
\quad
\begin{tikzpicture}[scale=0.5]
\graycell001
\cell10{$\bar3$}\cell11{$\bar1$}\cell123
\cell201\cell212\cell22{$\bar1$}
\node at (-2,-0.5) {(b)};
\end{tikzpicture}
\quad
\begin{tikzpicture}[scale=0.5]
\cell002
\cell10{$\bar1$}\cell11{$\bar2$}\graycell121
\cell201\cell212\cell22{$\bar1$}
\node at (-2,-0.5) {(c)};
\end{tikzpicture}
\caption{(a) $\Phi$-trivial, (b) $\Phi$-nondegenerate, and (c) $\Phi$-degenerate. The grey box marks the distinguished cell if such exists. The degenerate segment in (c) is composed of the two cells at the tops of columns 2 and 3 in row 2.}\label{fig:degenerate}
\end{figure}

The main result in this section will be the following theorem.

\begin{theorem}\label{thm:A2FP}
\begin{align*}
C_{\lambda}[X(q-1);q,t]& = \sum_{\substack{\sigma\in\widetilde{\PQT}(\lambda)\\\sigma\ \text{is}\ \Phi\text{-trivial}}} (-1)^{m(\sigma)} q^{p(\sigma)+\maj(\sigma)}t^{\quinv(\sigma)}x^{|\sigma|}\\
&=\sum_{\mu \leq \lambda'} d_{\mu\lambda}(q,t)m_{\mu}.
\end{align*}
\end{theorem}

The outline of our proof is as follows. First we define an involution on $\Phi$-nondegenerate fillings which will cancel the terms coming from those fillings in \eqref{C2}. This mirrors the corresponding involution from \cite[Section 5.2]{HHL05}. Next, in \cref{sec:Phi-degenerate} we will prove the existence of a bijection on $\Phi$-degenerate fillings, which will cancel the terms coming from those fillings in \eqref{C2}. Only terms arising from $\Phi$-trivial fillings remain, from which the proposition easily follows.

Recall the map $\Phi_u$ from \cref{def:Phi_u}. As we will see in the following lemma, under total ordering $<_2$, for certain $u$'s, the action of $\Phi_u$ on $\sigma$ has the property that it increases $\maj$ by exactly one, and all triples that are not contained in the degenerate segment contribute to $\quinv(\sigma)$ if and only if they contribute to $\quinv(\Phi_u(\sigma))$.

When $\sigma$ is not $\Phi$-trivial, the following two lemmas present two important properties of $\Phi_u(\sigma)$ that hold for certain choices of the cell $u$.

\begin{lemma}\label{lem:Phi_u:maj}
Suppose $\sigma$ is not $\Phi$-trivial with distinguished label $a$, let $r$ be the row containing the distinguished cell, and let $u$ be any cell in row $r$ with $|\sigma(u)|=a$. 
Then, 
\begin{equation}
\maj(\Phi_u(\sigma))=\maj(\sigma)+\begin{cases}1,&\mbox{if}\ \sigma(u)=a,\\-1,&\mbox{if}\ \sigma(u)=\bar{a}.\end{cases}\end{equation}
\end{lemma}

\begin{lemma}\label{lem:Phi_u:quinv}
Suppose $\sigma$ is not $\Phi$-trivial with distinguished label $a$, and let $r$ be the row containing the distinguished cell. Then
\begin{enumerate}
\item if $\sigma$ is $\Phi$-nondegenerate and $u$ is the distinguished cell, or
\item if $\sigma$ is $\Phi$-degenerate and $u$ is any cell in the degenerate segment of row $r$ with $|\sigma(u)|=a$, 
\end{enumerate}
the following holds: every non-degenerate triple containing $u$ in $\Phi_u(\sigma)$  is a $\quinv$ triple if and only if it's also a $\quinv$ triple in $\sigma$.
\end{lemma}

\begin{example}
Let $\sigma$ be the filling in \cref{fig:degenerate}(b). The distinguished label is $a=1$ and $u=(3,1)$ is the distinguished cell (i.e. the $1$ in the topmost row), then $\Phi_u(\sigma)$ is the filling

\begin{center}
\begin{tikzpicture}[scale=0.5]
\cell00{$\bar1$}
\cell10{$\bar3$}\cell11{$\bar1$}\cell123
\cell201\cell212\cell22{$\bar1$}
\end{tikzpicture}
\end{center}
Moreover, $\maj(\sigma)=3$, $\maj(\Phi_u(\sigma))=4$, and $\quinv(\sigma)=\quinv(\Phi_u(\sigma))=3$.
\end{example}

\begin{remark} 
The knowledgeable reader may observe that a similar involution $\Phi_{u'}$ was
defined in \cite[Section 5.2]{HHL05}, where $u'$ was chosen to be the topmost, \emph{leftmost} cell of $\dg(\lambda)$ whose absolute value matches the distinguished label of $\sigma$. 
\end{remark}

We make a key observation. 
\begin{lemma}\label{lem:obs}
Let $a$ be the distinguished label and $r$ the row containing the distinguished cell in a filling $\sigma$. Let $x$ be any cell in either row $r-1$ or $r$, and let $y$ be any cell in row $r+1$ or higher. Then $|\sigma(x)|\geq a$ and $|\sigma(y)|>a$.
\end{lemma}
\begin{proof}
Since $a$ is by definition the smallest positive integer such that there exists a cell in row $j>a$ with absolute value $a$, a row greater than or equal to $r-1$ cannot contain any cells with absolute value strictly less than $a$. Moreover, since $u$ is chosen such that $r$ is maximal, rows $r+1$ and higher cannot contain any cells with absolute value less than or equal to $a$. 
\end{proof}

\begin{proof}[Proof of \cref{lem:Phi_u:maj}]
The proof is identical to that of \cite[Lemma 5.2]{HHL05}. Let $u$ be any cell in row $r$ with $|\sigma(u)|=a$. Without loss of generality, assume $\sigma(u)=a$. The only possible change in the descent set of $\sigma$ by the action of $\Phi_u$ can occur at the cell $u$ or the cell directly above $u$ if such exists. 
Consider the cells $x$ and $y$ with contents $c:=\sigma(x)$ and $b:=\sigma(y)$ directly above and below $u$, respectively, if such exist. 
Then we obtain a transition exactly analogous to \eqref{psi-action-maj}.

By \cref{lem:obs}, $|b| \geq a$ and $|c| > a$. Thus $I_2(c,a)=I_2(\bar{a},b)=1$ and $I_2(c,\bar{a})=I_2(a,b)=0$, so necessarily $x\in\Des(\sigma)$ and $u\not\in\Des(\sigma)$, while $x\not\in\Des(\Phi_u(\sigma))$ and $u\in\Des(\Phi_u(\sigma))$, implying that $\maj(\Phi_u(\sigma))=\maj(\sigma)+1$. A similar argument proves the statement in the cases where the cells $x$ or $y$ may not exist.
\end{proof}

\begin{proof}[Proof of \cref{lem:Phi_u:quinv}]
We adapt the proof of \cite[Lemma 5.2]{HHL05} to the setting of $\widetilde{\PQT}$.  Let $u$ be a cell in row $r$ with $|\sigma(u)|=a$, such that either
\begin{enumerate}
\item $\sigma$ is $\Phi$-nondegenerate and $u$ is the distinguished cell, or
\item $\sigma$ is $\Phi$-degenerate and $u$ is in the degenerate segment of row $r$. 
\end{enumerate}
Without loss of generality, assume $\sigma(u)=a$. 

Clearly the only triples that may be affected by $\Phi_u$ are those containing the cell $u$. Let us inspect the three possible \emph{nondegenerate} configurations of such triples. Let $x,y$ be the other two cells in the triple, with contents $b:=\sigma(y)$ and $c:=\sigma(x)$; since the triples we examine are non-degenerate, both $x$ and $y$ must exist. We invoke \cref{lem:obs}.
\begin{flushleft}
\begin{tikzpicture}[scale=0.5]
\node at (-4.5,0) {\emph{Case 1}:};
\cellL0{.5}{$a$}{$u$}\cellL1{.5}{$c$}{$x$}\cellL1{3.5}{$b$}{$y$}
\node at (1.3,-.5) {$\cdots$};
\node at (5.5,0) {$\longrightarrow$};
\node at (5.5,.7) {$\Phi_u$};
\cellL0{8.5}{$\bar{a}$}{$u$}\cellL1{8.5}{$c$}{$x$}\cellL1{11.5}{$b$}{$y$}
\node at (9.3,-.5) {$\cdots$};
\end{tikzpicture}
\end{flushleft}
Both $|b|\geq a$ and $|c|\geq a$. Then $I_2(a,c)=I_2(a,b)=0$ and $I_2(\bar{a},c)=I_2(\bar{a},b)=1$. Thus the triple $(u,x,y)$ is a $\quinv$ triple both in $\sigma$ and in $\Phi(\sigma)$ if and only if $I_2(b,c)=1$.
\begin{flushleft}
\begin{tikzpicture}[scale=0.5]
\node at (-4.5,0) {\emph{Case 2}:};
\cellL0{.5}{$c$}{$x$}\cellL1{.5}{$b$}{$y$}\cellL1{3.5}{$a$}{$u$}
\node at (1.3,-.5) {$\cdots$};
\node at (5.5,0) {$\longrightarrow$};
\node at (5.5,.7) {$\Phi_u$};
\cellL0{8.5}{$c$}{$x$}\cellL1{8.5}{$b$}{$y$}\cellL1{11.5}{$\bar{a}$}{$u$}
\node at (9.2,-.5) {$\cdots$};
\end{tikzpicture}
\end{flushleft}
Since $|c|>a$, we have $I_2(c,a)=1$ and $I_2(c,\bar{a})=0$. Since $|b|\geq a$, we have $I_2(a,b)=0$ and $I_2(\bar{a},b)=1$. Thus the triple $(x,y,u)$ is a $\quinv$ triple both in $\sigma$ and in $\Phi(\sigma)$ if and only if $I_2(c,b)=0$.
\begin{flushleft}
\begin{tikzpicture}[scale=0.5]
\node at (-4.5,0) {\emph{Case 3}:};
\cellL0{.5}{$b$}{$y$}\cellL1{.5}{$a$}{$u$}\cellL1{3.5}{$c$}{$x$}
\node at (1.3,-.5) {$\cdots$};
\node at (5.5,0) {$\longrightarrow$};
\node at (5.5,.7) {$\Phi_u$};
\cellL0{8.5}{$b$}{$y$}\cellL1{8.5}{$\bar{a}$}{$u$}\cellL1{11.5}{$c$}{$x$}
\node at (9.2,-.5) {$\cdots$};
\end{tikzpicture}
\end{flushleft}
This case can only occur if $\sigma$ is $\Phi$-nondegenerate, since otherwise $u$ is necessarily in the degenerate segment, contradicting the existence of the cell $y$. Thus we assume $u$ is the distinguished cell, which means it is the first cell in reading order with content equal to $a$ or $\bar{a}$. Consequently $|c|>a$ because $x$ precedes $u$ in the reading order, and so $I_2(c,a)=1$ and $I_2(c,\bar{a})=0$. Since $|b|>a$, we have $I_2(b,a)=1$ and $I_2(b,\bar{a})=0$. Thus, again, the triple $(y,u,x)$ is a $\quinv$ triple both in $\sigma$ and in $\Phi(\sigma)$ if and only if $I_2(b,c)=1$.

In each of these cases, all non-degenerate triples containing $u$ contribute to $\quinv(\sigma)$ if and only if they also contribute to $\quinv(\Phi_u(\sigma))$, as desired.
\end{proof}

\begin{corollary}\label{cor:nondegenerate}
Let $\sigma\in\widetilde{\PQT}(\lambda)$ be $\Phi$-nondegenerate with distinguished cell $u$.
\begin{enumerate}[label=(\roman*)]
\item $\Phi_u$ is an involution, and
\item $q^{p(\Phi_u(\sigma))+\maj(\Phi_u(\sigma))}t^{\quinv(\Phi_u(\sigma))} = q^{p(\sigma)+\maj(\sigma)}t^{\quinv(\sigma)}$.
\end{enumerate}
\end{corollary}

\begin{proof}
The fact that $\Phi_u$ is an involution is immediate, since $|\Phi_u(\sigma)|=|\sigma|$, and so both fillings have the same distinguished label and the same distinguished cell $u$ is chosen for $\Phi_u(\sigma)$ as for $\sigma$.

Part (ii) is immediate from \cref{lem:Phi_u:maj,lem:Phi_u:quinv}. Without loss of generality, assume $\sigma(u)\in\mathbb{Z}_+$. Then $p(\Phi_u(\sigma))=p(\sigma)-1$, $\maj(\Phi_u(\sigma)) = \maj(\sigma)+1$, and $\quinv(\sigma)=\quinv(\Phi_u(\sigma))$ since $\sigma$ is $\Phi$-nondegenerate, so the equality follows.
\end{proof}

Unfortunately, when $\sigma\in\widetilde{\PQT}(\lambda)$ is $\Phi$-degenerate, it may be the case that no choice of the cell $u$ for which $|\sigma(u)|=a$ will be such that $\Phi_u$ preserves $\quinv$ and changes $\maj$ by 1. For example, going back to the filling $\sigma$ in \cref{fig:degenerate}(c), which has $\maj=3$ and $\quinv=3$, we list below the outputs of $\Phi_u(\sigma)$ where we set $u$ to be each of the four cells whose content is $1$ or $\bar{1}$, with the highlighted cell indicating the choice of $u$.

\begin{center}
\begin{tikzpicture}[scale=0.5]
\cell002
\graycell10{$1$}\cell11{$\bar2$}\cell12{1}
\cell201\cell212\cell22{$\bar1$}
\node at (0.5,-3) {$\substack{\maj=2\\\quinv=2}$};
\end{tikzpicture}
\quad
\begin{tikzpicture}[scale=0.5]
\cell002
\cell10{$\bar1$}\cell11{$\bar2$}\graycell12{$\bar1$}
\cell201\cell212\cell22{$\bar1$}
\node at (0.5,-3) {$\substack{\maj=4\\\quinv=4}$};
\end{tikzpicture}
\quad
\begin{tikzpicture}[scale=0.5]
\cell002
\cell10{$\bar1$}\cell11{$\bar2$}\cell121
\graycell20{$\bar1$}\cell212\cell22{$\bar1$}
\node at (0.5,-3) {$\substack{\maj=3\\\quinv=2}$};
\end{tikzpicture}
\quad
\begin{tikzpicture}[scale=0.5]
\cell002
\cell10{$\bar1$}\cell11{$\bar2$}\cell121
\cell201\cell212\graycell22{$1$}
\node at (0.5,-3) {$\substack{\maj=3\\\quinv=2}$};
\end{tikzpicture}
\end{center}

Thus we need to do something more subtle. In the following section, we will describe with some technical manipulations how to define a bijection on $\Phi$-degenerate fillings that preserves $\quinv$ and changes $\maj$ by 1.

\section{A bijection on $\Phi$-degenerate tableaux}
\label{sec:Phi-degenerate}

In this section, we \emph{demonstrate the existence} of a bijection $\Theta$ on $\Phi$-degenerate super fillings with the following properties. Let $\sigma\in\widetilde{\PQT}$. 
\begin{enumerate}[label=\roman*.]
\item $x^{|\Theta(\sigma)|}=x^{|\sigma|}$ and $p(\Theta(\sigma))=p(\sigma)  \pm 1$,
\item $\maj(\Theta(\sigma))=\maj(\sigma)\pm 1$, and
\item $\quinv(\Theta(\sigma))=\quinv(\sigma)$.
\end{enumerate}
Note that we do not succeed in constructing such an bijection, and it remains an open question to find an explicit bijection.

\begin{defn}For a $\Phi$-degenerate filling $\sigma$, define the \emph{degenerate word} to be the word read off the content of the cells that form degenerate triples in the topmost row containing
the distinguished cell. For example, in \cref{fig:degenerate}(c), the degenerate word is $(\bar{2},1)$.
\end{defn}

The main result of this section is the following.

\begin{theorem}
\label{prop:Phi-degenerate}
For a partition $\lambda$, the sum over the $\Phi$-degenerate subset of fillings in $\widetilde{\PQT}(\lambda)$ in the right hand side of \eqref{C2}, is zero.
\[
\sum_{\substack{\sigma\in\widetilde{\PQT}(\lambda)\\\sigma\ \text{is}\ \Phi\text{-degenerate}}} (-1)^{m(\sigma)} q^{p(\sigma)+\maj(\sigma)}t^{\quinv(\sigma)}x^{|\sigma|} =0.
\]
\end{theorem}

The outline of our proof is as follows. Let $a$ be the distinguished label of a $\Phi$-degenerate filling $\sigma\in\widetilde{\PQT}$. We will show in \cref{thm:Wbijection} that there exists a $\quinv$-preserving map $\phi:\widetilde{\PQT}\rightarrow \widetilde{\PQT}$ that preserves the content of the filling in absolute value (in particular $\phi$ preserves the content in all the letters that are not $a$ or $\bar{a}$), such that $\phi(\sigma)$ has exactly one more or one less $\bar{a}$ as $\sigma$. Next in \cref{def:tau} we define entry-swapping operators $\{\tau_j\}$ on $\widetilde{\PQT}$ that swap the topmost entries in columns $j,j+1$ (i.e. they swap entries in the degenerate word) while preserving the $\maj$ and $\quinv$ in the rest of the filling (\cref{lem:tau:maj,lem:tau:quinv}). There is then a unique product of operators $\tau_w=\tau_{i_1}\tau_{i_2}\ldots\tau_{i_s}$ that can be applied to $\Phi_u(\sigma)$ such that
\begin{itemize}
\item $u$ is a cell from the degenerate word that is designated by $\phi$,
\item the degenerate word of $\tau_w(\Phi_u(\sigma))$ matches $\phi(w)$, and 
\item $\hat{\sigma} = \tau_w(\Phi_u(\sigma))$ satisfies 
\[q^{p(\hat{\sigma})+\maj(\hat{\sigma})}t^{\quinv(\hat{\sigma})} = q^{p(\sigma)+\maj(\sigma)}t^{\quinv(\sigma)}.\]
\end{itemize}

\subsection{The operators $\tau_j$}

We first define the entry-swapping operator $\tau_j$ that swaps the topmost entries (if they are distinct) in columns $j,j+1$ of equal height while preserving certain statistics. This operator is modified for use in our setting from its original definition in \cite{LoehrNiese}.

\begin{defn}[Definition of the operator $\tau_j$]\label{def:tau}

Suppose two columns $j$ and $j+1$ have equal height, i.e.\
$\lambda_j=\lambda_{j+1}=k$. We define an operator $\tau_j$ which
exchanges contents of certain cells between columns $j$ and $j+1$. 

Write $\sigma(r,j)=a_r$ and $\sigma(r,j+1)=b_r$ for $r=1,\dots, k$. 

Define $r_{\max}$ to be the largest $r\in\{2,\dots,k\}$ with the following property:
either $(a_r, a_{r-1}, b_{r-1})$ and $(b_r, a_{r-1}, b_{r-1})$ are both in $\cQ$,
or both are not in $\cQ$. (That is, we look for the largest value of $r$ 
such that exchanging the entries $a_r$ and $b_r$
makes no difference to whether $((r,j), (r-1,j), (r-1, j+1))$ is a quinv triple.)
If there is no such $r$, let $r_{\max}=1$. 

Now let the operator $\tau_j$ swap the entries between columns $j$ and $j+1$
in rows $i$ with $r_{\max}\leq i\leq k$; i.e.\ for those values of $i$,
\[
\tau_j(\sigma)(i,j)=b_i,
\qquad 
\tau_j(\sigma)(i,j+1)=a_i,
\]
while all other entries are the same in $\sigma$ and in $\tau_j(\sigma)$,
as in the picture below, where we denote $\ell:=r_{\max}$.

\begin{equation*}
\renewcommand{\arraystretch}{1}  
\setlength{\arraycolsep}{1pt} 
\begin{array}{l|c|c|}
\multicolumn{1}{c}{}&
\multicolumn{2}{c}{\sigma}
\\
\cline{2-3} 
\text{row } k &  a_k & b_k\\
 \cline{2-3} 
\,\,\,\vdots & \vdots & \vdots \\
\cline{2-3}
\text{row } \ell & a_{\ell} & b_{\ell}\\
\cline{2-3} 
\text{row } \ell-1\,\,\,\,\,\, & a_{\ell-1} & b_{\ell-1}\\
 \cline{2-3}
\,\,\,\vdots & \vdots & \vdots \\
 \cline{2-3}
 \multicolumn{1}{c}{}&
\multicolumn{1}{c}{\scriptstyle j}&
\multicolumn{1}{c}{\scriptstyle j+1}
\end{array}
\,\,\xrightarrow{\makebox[1.5cm]{$\tau_j$}}\,\,
\begin{array}{|c|c|}
\multicolumn{2}{c}{\tau_j(\sigma)}
\\
\hline
b_k &  a_k\\
\hline
\vdots & \vdots \\
\hline
b_{\ell} & a_{\ell}\\
\hline
a_{\ell-1} & b_{\ell-1}\\
\hline
\vdots & \vdots \\
\hline
\multicolumn{1}{c}{\scriptstyle j}&
\multicolumn{1}{c}{\scriptstyle j+1}
\end{array}
\end{equation*}

Put another way: we swap the pair in row $k$, and iteratively, if we have swapped the pair in row $i$ and this made a difference to whether the triple $(i,j), (i-1, j), (i-1, j+1)$ is a quinv
triple, then we swap in row $i-1$ also. 
\end{defn}

\begin{example}\label{ex:tau}
Suppose $\sigma$ has columns $j,j+1$ as shown below. Then $k=5$ and $r_{\max}=3$ since both $(2,3,4)\in\cQ$ and $(3,3,4)\in\cQ$. Thus applying the operator $\tau_j$ gives the following. The cells whose content was swapped are shown in grey.

\begin{equation*}
\renewcommand{\arraystretch}{1}  
\setlength{\arraycolsep}{1pt} 
\begin{array}{|c|c|}
\multicolumn{2}{c}{\sigma}
\\
\hline
3 &  4\\
\hline
2 & 3  \\
\hline
2 & 3 \\
\hline
3 & 4\\
\hline
1 & 3 \\
\hline
\multicolumn{1}{c}{\scriptstyle \phantom{+} j\phantom{1}}&
\multicolumn{1}{c}{\scriptstyle j+1}
\end{array}
\,\,\xrightarrow{\makebox[1.5cm]{$\tau_j$}}\,\,
\begin{array}{|c|c|}
\multicolumn{2}{c}{\tau_j(\sigma)}
\\
\hline
\cellcolor{gray!50}4 & \cellcolor{gray!50} 3\\
\hline
\cellcolor{gray!50}3 & \cellcolor{gray!50}2  \\
\hline
\cellcolor{gray!50}3 & \cellcolor{gray!50}2 \\
\hline
3 & 4\\
\hline
1 & 3 \\
\hline
\multicolumn{1}{c}{\scriptstyle \phantom{+} j\phantom{1}}&
\multicolumn{1}{c}{\scriptstyle j+1}
\end{array}
\end{equation*}

\end{example}

It is straightforward to check by the definition that $\tau_j$ is an involution. Proving the following \cref{lem:tau:maj,lem:tau:quinv} regarding the properties of the action of $\tau_j$ on super fillings $\widetilde{\PQT}$ is also straightforward, but technical.

\begin{lemma}\label{lem:tau:maj}
Let $\sigma\in\widetilde{\PQT}$ and let $j$ be such that the columns $j,j+1$ are of equal height $k$. Then 
\begin{equation}
\maj(\tau_j(\sigma))=\maj(\sigma).
\end{equation}
\end{lemma}

\begin{lemma}\label{lem:tau:quinv}
Let $\sigma\in\widetilde{\PQT}$ and let $j$ be such that the columns $j,j+1$ are of equal height $k$. Let $x=(k,j)$ and $y=(k,j+1)$ be the cells at the tops of columns $j$ and $j+1$. Then
\begin{equation}
\quinv(\tau_j(\sigma))=\quinv(\sigma) + \begin{cases} 1,& \sigma(x)>\sigma(y),\\
-1,&\sigma(x)<\sigma(y),\\
0,&\sigma(x)=\sigma(y). \end{cases}
\end{equation}
\end{lemma}

\begin{remark}
We point out that our definition of $\tau_j$ is very similar to that of a similar operator in \cite[Definition 3.7]{compactformula}, with two important differences. First, $\tau_j$ necessarily swaps the topmost entries in the columns if they are different and acts trivially if they are identical, instead of swapping the first pair of non-identical entries (from the bottom) as in \cite{compactformula}. Second, in our case swaps propagate downwards from row $k$, whereas in the corresponding operator defined in \cite{compactformula}, the swaps propagate upwards. This has to do with the fact that a $\quinv$ triple has configuration 
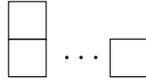
\[
\qtrip{\ }{\ }{\ }
\]
whereas the corresponding triple studied in \cite{compactformula} has configuration 
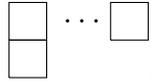
\[
\begin{tikzpicture}[scale=0.5]\cell10{\ }\cell20{\ } \cell1{2.7}{\ }
\node at (1,-0.5) {$\cdots$};
\end{tikzpicture}.
\]
\end{remark}

\begin{proof}[Proof of \cref{lem:tau:maj}] The proof is almost identical to that of \cite[Lemma 3.10]{compactformula}, but we present it here as a warm-up for \cref{lem:tau:quinv}, in the more nuanced setting of super fillings.

If $k=1$, there is no contribution to $\maj$ from columns $j,j+1$, so assume $k>1$. Consider the entries in columns $j,j+1$ in the two rows $\ell,\ell-1$, such that $1 < \ell \leq k$, as shown below.

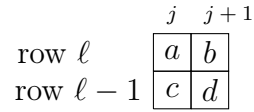
\begin{center}
\begin{tikzpicture}[scale=.5]
\cell01{$a$} \cell02{$b$}
\cell11{$c$} \cell12{$d$}

\node at (-2.65,.5) {\rm row $\ell$};
 \node at (-2,-.5) {\rm row $\ell-1$};
\node at (.5, 1.5) {{\tiny$j$}};
\node at (2, 1.5) {{\tiny$j+1$}};
\end{tikzpicture}
\end{center}

If $a,b$ are not swapped by $\tau_j$, there is nothing to check. If both $a,b$ and $c,d$ are swapped by $\tau_j$, the contribution to $\maj$ from the cells $(\ell,j),(\ell,j+1)$ remains the same. Thus the only case we need to check is when $\ell=r_{\max}$ so that the swapping procedure terminates at row $\ell$ and $a,b$ are swapped, but $c,d$ are not. Thus assume $\tau_j$ sends \raisebox{.1in}{\tableau{a&b\\c&d}} to  \raisebox{.1in}{\tableau{b&a\\c&d}}.

Since $\ell=r_{\max}$, we have that the triples $(a,c,d)$ and $(b,c,d)$ are are either both in $\cQ$, or neither is in $\cQ$. Following the condition in \cref{def:super} of a triple to be in $\cQ$, let us consider the sets
\[ 
ACD = \left\{\begin{matrix}I(a,c)=1\\I(a,d)=0\\I(d,c)=0 \end{matrix} \right\} 
\quad \mbox{and} \quad 
BCD = \left\{\begin{matrix}I(b,c)=1\\I(b,d)=0\\I(d,c)=0 \end{matrix} \right\},
\]
where we know that either both $ACD$ and $BCD$ contain exactly one condition that is true, or they both contain exactly two that are true. Recall that a pair of cells $u,v$ in $\sigma$ with $v=\South(u)$ forms a descent if $I(\sigma(u),\sigma(v))=1$. We consider two cases based on whether or not $I(d,c)=0$.

\emph{Case 1: Suppose $I(d,c)=0$.} If the two triples are in $\cQ$, both of the remaining conditions in $ACF$ and $BCF$ must be false, and so we must have $I(a,c)=I(b,c)=0$ and $I(a,d)=I(b,d)=1$. This implies the cells $(\ell,j),(\ell,j+1)$ make the same contribution to $\maj$ in $\sigma$ and $\tau_j(\sigma)$. On the other hand, if the two triples are not in $\cQ$, exactly one of the two remaining conditions in both $ACD$ and $BCD$ must be true, and so we must have $I(a,c)=I(a,d)$ and $I(b,c)=I(b,d)$. This also results in the same contribution to $\maj$ from the cells $(\ell,j),(\ell,j+1)$ in $\sigma$ and $\tau_j(\sigma)$.

\emph{Case 2: Suppose $I(d,c)=1$.} This case is similar. If the two triples are in $\cQ$, we must have $I(a,c)=I(a,d)$ and $I(b,c)=I(b,d)$ since exactly one of the two remaining conditions must be true. If the two triples are not in $\cQ$, both of the two remaining conditions must be true, and so we must have $I(a,c)=I(b,c)=1$ and $I(a,d)=I(b,d)=0$. Both cases result in the same contribution to $\maj$ from the cells $(\ell,j),(\ell,j+1)$.

We see finally that even though the locations of the descents may change after applying $\tau_j$, the total contribution to $\maj$ in columns $j,j+1$ remains constant.
\end{proof}

\begin{proof}[Proof of \cref{lem:tau:quinv}]
The proof follows by a similar argument as that of \cite[Lemma 3.11]{compactformula}, though the details are a bit nuanced when we work with super fillings. 

Consider the columns $j,j+1$ of $\sigma$ shown below with $x=(k,j)$, $y=(k,j+1)$, and $a_k=\sigma(x)$, $b_k=\sigma(y)$. Assume $a_k\neq b_k$, since otherwise $\tau_j$ is trivial. Let $\ell:=r_{\max}$, so $\tau_j$ swaps all pairs of entries in rows $j,j+1$ from row $k$ down to row $\ell$. If $\ell=1$, then the columns $j,j+1$ simply switched places, which trivially proves the claim. Thus assume $\ell>1$, and let $c,d$ be the cells in row $\ell-1$ in columns $j,j+1$ respectively. 

\begin{center}
\begin{tikzpicture}[scale=0.55]
\node at (2,1.7){$\sigma$};
\node at (9,1.7){$\tau_j(\sigma)$};
\node at (-3.05,0.5){row $k$};
\node at (-3.1,-1.5){row $\ell$};
\node at (-2.5,-2.5){row $\ell-1$};
\draw[->](5,-1)--(6.5,-1) node[midway,above] {$\tau_j$};
\cell01{$a_k$}\cell02{$b_k$}\cell0{4.5}{$f_1$}
\cell11{$\vdots$}\cell12{$\vdots$}\cell1{4.5}{$\vdots$}
\cell21{$a_{\ell}$}\cell22{$b_{\ell}$}\cell2{4.5}{$f_{\ell}$}
\cell31{$c$}\cell32{$d$}\cell3{4.5}{$f$}

\cell08{$b_k$}\cell0{9}{$a_k$}\cell0{11.5}{$f_1$}
\cell18{$\vdots$}\cell1{9}{$\vdots$}\cell1{11.5}{$\vdots$}
\cell28{$b_{\ell}$}\cell2{9}{$a_{\ell}$}\cell2{11.5}{$f_{\ell}$}
\cell38{$c$}\cell3{9}{$d$}\cell3{11.5}{$f$}

\node at (0.5,-3.5){{\tiny $j$}};
\node at (2,-3.5){{\tiny $j+1$}};
\node at (4,-3.5){{\tiny $s$}};
\node at (7.5,-3.5){{\tiny $j$}};
\node at (9,-3.5){{\tiny $j+1$}};
\node at (11,-3.5){{\tiny $s$}};
\end{tikzpicture}
\end{center}

Since only columns $j,j+1$ are affected by $\tau_j$, we only need to examine the contribution in $\quinv(\sigma)$ and $\quinv(\tau_j(\sigma))$ of triples with at least one entry in one of those columns. By construction, with the exception of the degenerate triple formed by the entries $(a_k,b_k)$, all triples in $\sigma$ in columns $j,j+1$ make the same contribution in $\quinv(\sigma)$ as the corresponding triples in the same locations of $\tau_j(\sigma)$ do in $\quinv(\tau_j(\sigma))$. Moreover, triples containing cells in columns to the left of $j$ will not be affected either, nor will any triples with all cells below row $\ell$. It remains for us to examine the triples containing cells in rows greater than or equal to $\ell-1$ in one of the columns $j,j+1$, and a cell in a column $s$ for some $s>j+1$. Both $\sigma$ and $\tau_j(\sigma)$ have the triples $((i+1,j),(i,j),(i,s))$ and $((i+1,j+1),(i,j+1),(i,s))$ with the same contents $(a_{i+1},a_i,f_i)$ and $(b_{i+1},b_i,f_i)$ for $\ell\leq i \leq k$ (per our convention we set $\sigma(a_{k+1})=\sigma(b_{k+1})=0$). Therefore the only triples whose content was altered are the triples $((\ell,j),(\ell-1,j),(\ell-1,s))$ and $((\ell,j+1),(\ell-1,j+1),(\ell-1,s))$ with respective contents $(a_{\ell},c,f)$ and $(b_{\ell},d,f)$ in $\sigma$, which ended up with respective content $(b_{\ell},c,f)$ and $(a_{\ell},d,f)$ in $\tau_j(\sigma)$. We will now check that the total number of $\quinv$ triples is preserved in each of these pairs. To simplify notation, we will drop the subscripts and write $a=a_{\ell}$ and $b=b_{\ell}$.

We begin with a key observation. Since $\ell=r_{\max}$, we have that either both $(a,c,d)$ and $(b,c,d)$ are in $\cQ$, or neither is in $\cQ$. Equivalently, this means that in the sets
\[ 
ACD = \left\{\begin{matrix}I(a,c)=1\\I(a,d)=0\\I(d,c)=0 \end{matrix} \right\} 
\quad \mbox{and} \quad 
ABD = \left\{\begin{matrix}I(b,c)=1\\I(b,d)=0\\I(d,c)=0 \end{matrix} \right\},
\]
either exactly one condition is true in both, or exactly two are true in both. By studying the cases in the proof of \cref{lem:tau:maj}, we get that either both $I(a,c)=I(a,d)$ and $I(b,c)=I(b,d)$, or both $I(a,c)\neq I(a,d)$ and $I(b,c)\neq I(b,d)$. 

Let us now consider the triples $(a,c,f)$, $(b,d,f)$, and $(b,c,f)$, $(a,d,f)$, respectively corresponding to the sets
{\scriptsize 
\[ 
ACF = \left\{\begin{matrix}I(a,c)=1\\I(a,f)=0\\I(f,c)=0 \end{matrix} \right\},\  
BDF = \left\{\begin{matrix}I(b,d)=1\\I(b,f)=0\\I(f,d)=0 \end{matrix} \right\} 
\; \mbox{v/s} \;
BCF = \left\{\begin{matrix}I(b,c)=1\\I(b,f)=0\\I(f,c)=0 \end{matrix} \right\},\ 
ADF = \left\{\begin{matrix}I(a,d)=1\\I(a,f)=0\\I(f,d)=0 \end{matrix} \right\}.
\]
}

We split our argument into two cases. To eliminate extra notation, let us identify the triples with their corresponding sets. 

In the first case, let $I(a,c)=I(a,d)$ and $I(b,c)=I(b,d)$. If $I(f,c)=I(f,d)$, then $ACF$ (\emph{resp.} $BDF$) has the same number of true conditions as $ADF$ (\emph{resp.} $BCF$), which means the pairs make the same contribution to $\quinv(\sigma)$ as to $\quinv(\tau_j(\sigma))$. Thus let us assume $I(f,c)\neq I(f,d)$. In the table below, we examine the cases under the condition that there is no contradiction in either $ACF$ or $ADF$ arising from the latter having either all conditions true or all conditions false. 
Our notation is as follows: when we write, say, $\text{Set} \Rightarrow X$, this means that $X$ must be true in order to avoid a contradiction in $\text{Set}$.

\begin{center}
{\renewcommand{\arraystretch}{1.2}
\flushleft
\footnotesize
\begin{tabular}{l|c|c|}
& $I(f,c)=0$ and $I(f,d)=1$ & $I(f,c)=1$ and $I(f,d)=0$\\ \hline
$I(a,c)=I(a,d)=1$ & $ACF \Rightarrow I(a,f)=1$ & $ADF \Rightarrow I(a,f)=1$ \\ \hline
$I(a,c)=I(a,d)=0$ & $ADF \Rightarrow I(a,f)=0$ & $ACF \Rightarrow I(a,f)=0$\\ \hline

\end{tabular}
}
\end{center}
 
This table implies that necessarily, $I(a,c)=I(a,d)=I(a,f)$.
An analogous argument yields that necessarily, $I(b,d)=I(b,f)=I(b,c)$. From here it is straightforward to check that for any choice of $I(f,c)$ and $I(f,d)$, $ACF$ and $BCF$ make the same contribution to $\quinv(\sigma)$, as BDF and ADF to $\quinv(\tau_j(\sigma))$, and so the total contribution of the pair $ACF, BDF$ to $\quinv(\sigma)$ matches the contribution of the pair $BCF, ADF$ to $\quinv(\tau_j(\sigma))$.  

In the second case, let $I(a,c)\neq I(a,d)$ and $I(b,c)\neq I(b,d)$. Because $\maj$ is preserved, we must also have that $I(a,c)\neq I(b,d)$, implying that $I(a,c)=I(b,c)$ and $I(b,d)=I(a,d)$. We must also consider all choices for $I(f,c)$ and $I(f,d)$. We complete the argument by carefully considering all possible cases in the table below. We use some additional shorthand notation: when we write, say, $\{\cancel{ACF}, ADF\}$ this means $ACF$ is not a $\quinv$ triple, but $ADF$ is one. 

{\renewcommand{\arraystretch}{1.2}
\flushleft
\footnotesize
\begin{tabular}{l|c|c|c|c|}
& $I(f,c)=0,$ & $I(f,c)=0,$ & $I(f,c)=1,$ & $I(f,c)=1,$\\ 
& $I(f,d)=0$ & $I(f,d)=1$ & $I(f,d)=0$ & $I(f,d)=1$\\ \hline
$I(a,c)=1$, & $ACF \Rightarrow I(a,f)=1$ & $ACF \Rightarrow$ & $ACF=ADF$ & $ADF \Rightarrow I(a,f)=0$  \\
$I(b,c)=1$,& $\Rightarrow\{\cancel{ACF},\ ADF\}$ & $I(a,f)=1$, &dep.~on~$I(a,f)$ & $\Rightarrow\{\cancel{ACF}, ADF\}$  \\
$I(b,d)=0$, & $BCF \Rightarrow I(b,f)=1$ & contradicts & $BCF=BDF$  & $BDF \Rightarrow I(b,f)=0$  \\ 
$I(a,d)=0$& $\Rightarrow \{BDF,\ \cancel{BCF}\}$ & $ADF$ &dep.~on~$I(b,f)$ & $\Rightarrow \{BDF,\ \cancel{BCF}\}$  \\\hline
$I(a,c)=0$, & $ADF \Rightarrow I(a,f)=1$ & $ACF=ADF$ & $ACF \Rightarrow$ & $ACF \Rightarrow I(a,f)=0$  \\
$I(b,c)=0$,& $\Rightarrow \{ACF,\ \cancel{ADF}\}$ &dep.~on~$I(a,f)$ & $I(a,f)=0$, & $\Rightarrow \{ACF, \cancel{ADF}\}$  \\
$I(b,d)=1$, & $BDF \Rightarrow I(b,f)=1$ & $BCF=BDF$ & contradicts  & $BCF \Rightarrow I(b,f)=0$  \\ 
$I(a,d)=1$& $\Rightarrow\{\cancel{BDF},\ BCF\}$ &dep.~on~$I(b,f)$ & 
$ADF$ & $\Rightarrow\{\cancel{BDF},\ BCF\}$  \\\hline
\end{tabular}
}

\medskip
Again, in every possible case, the contribution of the pair $ACF, BDF$ to $\quinv(\sigma)$ matches the contribution of the pair $BCF, ADF$ to $\quinv(\tau_j(\sigma))$. Thus we conclude that the only triple whose contribution to $\quinv(\sigma)$ is different from its contribution to $\quinv(\tau_j(\sigma))$ is the degenerate triple $(x,y)$ at the top of the columns $j,j+1$. If $\sigma(x)=\sigma(y)$, $\tau_j(\sigma)=\sigma$. If $\sigma(x)\neq \sigma(y)$ and if $(x,y)$ was a degenerate $\quinv$ triple before the swap, swapping them will cause $\tau_i(\sigma)$ to lose one $\quinv$ triple, and gain one otherwise. 
\end{proof}

\begin{remark}\label{rem:PDS}
We will be working with permutations, and corresponding reduced expressions of those permutations to products of the operators $\tau_j$. Unfortunately, the operators do not satisfy braid relations, i.e. in general $\tau_j\tau_{j+1}\tau_j(\sigma) \neq \tau_{j+1}\tau_j\allowbreak \tau_{j+1}(\sigma)$. Thus we will need to choose a canonical reduced expression for a given permutation to make our construction well-defined. One way to do this uses the \emph{positive distinguished subexpression}, or PDS, of a reduced expression, defined in \cite{MarshRietsch}. We denote by $\PDS(\pi)$ the PDS corresponding to a permutation $\pi$. We will refer the reader to \cite[Section 3]{compactformula} for a detailed treatment of this notion; in our setting it suffices to assert that $\PDS(\pi)$ is a unique reduced expression for any permutation $\pi$.
\end{remark} 

\section{From a bijection on words to a bijection on 
super-fillings}
\label{sec:bijPQT}

In this section, we prove \cref{thm:Wbijection}, which states the existence of a bijection on the set of $\Phi$-degenerate fillings $\widetilde{\PQT}$, using the super-alphabet $\mathcal{A}$ and total ordering $<_2$ on that alphabet. \cref{thm:Wbijection} is needed in the proof of \cref{prop:Phi-degenerate}.

We begin with some definitions. Let $n \geq 3$ be a positive integer and let $\alpha = (\alpha_2,\dots, \alpha_{n-1})$ be a tuple of positive integers. 
Denote $|\alpha| = \sum_i \alpha_i$ and let $N > |\alpha|$ be a positive integer. For convenience, let $L = N - |\alpha|$.
We will consider words of length $N$ in the alphabet $[n] = \{1,\dots, n\}$ as follows. For $0 \leq k \leq L$, define
\begin{equation}
\label{wk-def}
W_k \equiv W^{(\alpha,L)}_k = \left\{ w = (w_1,\dots,w_N) \left| 
\begin{matrix}
c_n(w) = k, c_1(w) = L-k,  \\
c_i(w) = \alpha_i \text{ for } 2 \leq i \leq n-1
\end{matrix}  
\right. \right\},
\end{equation}
where $c_i(w)$ counts the number of times the letter $i$ appears in $w$.

\begin{example}
\label{eg:Wk}
For example, with $n = N = 3, L = 2$ and $\alpha = (1)$, we have
\begin{equation}
\label{Wk-eg}
\begin{split}
W_0 &= \{112, 121, 211\}, \\
W_1 &= \{123, 132, 213, 231, 312, 321 \}, \\
W_2 &= \{233, 323, 332 \}.
\end{split}
\end{equation}

\end{example}

Recall that a {\em coinversion} of a word $w$ is a pair $(i,j)$, $i < j$, such that $w_i < w_j$, and $\coinv(w)$ as the number of such pairs. In our applications, we will be mapping finite subsets of the super-alphabet $\mathcal{A}=\{1,2,\ldots,\bar{2},\bar{1}\}$ under total ordering $<_2$ to $\{1,2,\ldots,n\}$, such that the order is preserved. Fix $\ell$, 
$2 \leq \ell \leq n$. We will think of the letters $\ell, \dots, n$
as barred letters in $\mathcal{A}$.
For $w \in W_k$, define 
\begin{equation}
\label{word-quinv-def}
\quinv'(w) = \coinv(w) + \binom{\alpha_\ell}{2} + \cdots + \binom{\alpha_{n-1}}{2} +\binom{k}{2}.
\end{equation}
Thus if the preimage of $w$ in $\mathcal{A}$ corresponds to the degenerate word of some $\widetilde{\PQT}$ $\sigma$, $\quinv'(w)$ is equal to the contribution from the degenerate word to $\quinv(\sigma)$ (under the ordering $<_2$).

The proof of \cref{thm:Wbijection} relies on the existence of a $\quinv'$-preserving bijection $\phi$ on $W \equiv W^{(\alpha,L)}$ satisfying the desired properties. This result is stated below, but first we need some additional notation. For $0 \leq k \leq L$, partition $W_k$ from \eqref{wk-def} into two subsets. Let $W_k^\leq$ (resp. $W_k^>$) be those $w$ in $W_k$
whose position of the leftmost $n$ from the left is less than or equal to (resp. greater than) the position of the rightmost $1$ from the right ignoring all $n$'s. By convention, we will take $W_0 = W_0^>$ and $W_L = W_L^\leq$. For example, with $n = 3, N = 4, L = 3$ and $\alpha = (1)$,
\begin{align*}
W_2^\leq &= \{1 3 2 3 ,  1 3 3 2 ,  3 1 2 3,  3 1 3 2 ,  3 2 1 3 ,  3 2 3 1 ,  3 3 1 2 ,  3 3 2 1 \}, \\
W_2^> &= \{1233, 2133, 2313, 2331\}.
\end{align*}
Let $W^\leq = \bigcup_k W_k^\leq$ and $W^> = \bigcup_k W_k^>$. 

\begin{theorem}\label{thm:invol}
There exists a bijection $\phi : W^> \to W^\leq$ satisfying the
following conditions:
\begin{enumerate}
\item $\phi$ maps $W_k^>$ to $W_{k+1}^\leq$ bijectively for $0 \leq k \leq L-1$.
\item $\quinv'(w) = \quinv'(\phi(w))$ for all $w \in W^>$, and
\item The subword of $w$ in the letters $2, \dots, n-1$ is preserved by 
$\phi$.
\end{enumerate}
\end{theorem}

\cref{sec:bijwords} is devoted to the proof of \cref{thm:invol}, as it is quite technical.

Let $\sigma$ be a $\Phi$-degenerate filling with $a$ the distinguished label and distinguished word $w$. We split the set of $\Phi$-degenerate fillings into two disjoint sets that are also denoted by $W^>$ and $W^{\leq}$:
\begin{itemize}
\item $\sigma\in W^{>}$ if $w\in W^{>}$, and 
\item $\sigma\in W^{\leq}$ if $w\in W^{\leq}$. 
\end{itemize}

\begin{example}\label{ex:u_choice}
Consider the set of words 
\[
\{\bar1\bar12,\bar12\bar1, 2\bar1\bar1,\bar112,\bar121, 2\bar11,1\bar12,12\bar1,21\bar1,112,121,211\}.
\] 
This corresponds exactly to \cref{eg:W<>}
and hence, $W^{>}=\{211,121,2\bar11,112,21\bar1, \allowbreak 12\bar1\}$. The bijection $\phi$ from \cref{thm:invol} maps the elements of $W^{>}$ to the elements of $W^{\leq}$ as follows: 
\begin{align*}
\phi(211) = \bar121,\ \quinv'=0,&\quad \phi(121) = \bar112,\ \quinv'=1, \quad \phi(2\bar11) = \bar1\bar12,\ \quinv'=1,\\
\phi(112)=1\bar12,\ \quinv'=2,&\quad \phi(21\bar1)=\bar12\bar1,\  \quinv'=2,\quad\phi(12\bar1)=2\bar1\bar1,\ \quinv'=3.
\end{align*}
Consider the fillings $\sigma_1, \sigma_2$ shown below.
\[
\sigma_1=\raisebox{0.2in}{\tableau{2\\1&2&1&\bar1\\3&1&\bar2&\bar1\\2&1&2&1&\bar2}},
\quad
\sigma_2=\raisebox{0.2in}{\tableau{2\\1&1&\bar1&2\\3&1&\bar2&\bar1\\2&1&2&1&\bar2}}
\]
The filling $\sigma_1$ belongs to $W^{>}$ since $21\bar1\in W^{>}$, and the filling $\sigma_2$ belongs to $W^{\leq}$ since $1\bar12\in W^{\leq}$. 
\end{example}

We are now ready to state the main result of this section.

\begin{proposition}\label{thm:Wbijection}
Let $\sigma$ be a $\Phi$-degenerate filling with distinguished label $a$ and degenerate word $w$. Then there exists a map $\Theta$ on the subset $W^{>}$ of $\Phi$-degenerate fillings $\sigma$ such that
\begin{enumerate}
\item $\Theta\ :\ W^{>} \longrightarrow W^{\leq}$ is a bijection,
\item $\maj(\Theta(\sigma)) = \maj(\sigma) +1$,
\item $\quinv(\Theta(\sigma))=\quinv(\sigma)$,
\item $x^{|\Theta(\sigma)|}=x^{|\sigma|}$ and $p(\Theta(\sigma))=p(\sigma)-1$.
\end{enumerate}
\end{proposition}

\begin{proof}
Suppose $\sigma$ is a filling of $\dg(\lambda)$. We will define the map $\Theta$ acting on $\sigma\in W^{>}$ with the following steps.
\begin{enumerate}[label=(\roman*)]
\item\label{item:u} Set $u(\sigma)$ to be the first cell in reading order with content $a$ in $\sigma$. 
\item Let $\widehat{w}$ be the degenerate word of $\Phi_{u(\sigma)}(\sigma)$. ($\widehat{w}$ is equal to $w$, except that the rightmost $a$ is flipped to become $\bar{a}$.)
\item Let $\pi$ be the (unique) permutation acting on the right to get from $\widehat{w}$ to $\phi(w)$, i.e. $\widehat{w}\cdot \pi=\phi(w)$, and let $\PDS(\pi)=s_{i_1}s_{i_2}\ldots s_{i_m}$ be the unique reduced word which is the PDS corresponding to $\pi$ (see \cref{rem:PDS}).
\item Let $\ell$ be the number of columns in $\dg(\lambda)$ strictly to the left of the degenerate segment. Set 
$\Theta = \Phi_{u(\sigma)}\circ \tau_{i_1+\ell}\circ \tau_{i_2+\ell} \circ\cdots \circ \tau_{i_m+\ell}$ acting on the right, so that
\[
\Theta(\sigma) = \tau_{i_m+\ell}\circ \cdots\circ\tau_{i_2+\ell}\circ \tau_{i_1+\ell}\circ\Phi_{u(\sigma)}(\sigma).
\]
\end{enumerate}

We prove that $\Theta$ satisfies the properties (i)-(iv) in order. 

\emph{$\Theta$ satisfies (i):} The degenerate word of $\Theta(\sigma)$ is by construction equal to $\phi(w)$. Thus $w\in W^{>}$ implies that $\Theta(\sigma)\in W^{\leq}$. Next we note that the cell $u(\sigma)$ that is flipped in $\sigma\in W^{>}$ by $\Phi_{u(\sigma)}$ is determined uniquely, as is the product $ \tau_{i_m+\ell}\circ\cdots\circ \tau_{i_2+\ell} \circ \tau_{i_1+\ell}$ of operators applied to $\Phi_{u(\sigma)}(\sigma)$. The operators $\tau_j$ are involutions -- thus our construction is well-defined and an injection. Finally, by \cref{thm:invol} we have that for each $k\geq 0$, the number of elements in $W^{>}$ with exactly $k$ $\bar{a}$'s, the desired content, and a fixed number of $\quinv$ triples equals the number of elements in $W^{\leq}$ with exactly $k+1$ $\bar{a}$'s, from which we conclude that the map $\Theta$ we have defined is in fact a surjection. 

We define the reverse map $\Theta^{-1}$ from $W^{\leq}$ to $W^{>}$ by simply reversing the operators in the map $\Theta$. Let $\sigma\in W^{\leq}$ be a filling of $\dg(\lambda)$, and again let $a$ be the distinguished label and $w\in W^{\leq}$ the degenerate word of $\sigma$. Call $\widehat{u}(\sigma)$ the cell in the diagram $\dg(\lambda)$ that contains the rightmost entry $a$ in $\phi^{-1}(w)\in W^>$. Call $\widehat{\phi^{-1}(w)}$ the word obtained by flipping the rightmost $a$ in $\phi(w)$ to become $\bar{a}$. Then define the permutation $\pi$ such that $\widehat{\phi^{-1}(w)}\pi=w$, and let $\PDS(\pi)=s_{i_1}s_{i_2}\ldots s_{i_m}$. The reader may now check that the map from $\sigma\in W^{\leq}$ to $\Theta^{-1}(\sigma)$ given by
\[
\Theta^{-1}(\sigma) = \Phi_{\widehat{u}(\sigma)}\circ \tau_{i_1+\ell}\circ \tau_{i_2+\ell} \circ\cdots\circ \tau_{i_m+\ell}(\sigma),
\]
is the inverse of the map from $\sigma\in W^{>}$ to $W^{\leq}$, since $\widehat{u}(\sigma)$ is by design precisely equal to $u(\Theta^{-1}(\sigma))$ as defined in \cref{item:u}.

\emph{$\Theta$ satisfies (ii):}
By \cref{lem:Phi_u:maj}, when $u$ is any cell in the degenerate segment of $\sigma$ with $\sigma(u)=a$, the operator $\Phi_u$ acts on $\sigma$ by increasing the $\maj$ by exactly one. By \cref{lem:tau:maj} the operators $\tau_j$ leave the $\maj$ unchanged.

\emph{$\Theta$ satisfies (iii):} By \cref{lem:Phi_u:quinv}, when $u$ is any cell in the degenerate segment of $\sigma$ with $|\sigma(u)|=a$, all triples with the exception of degenerate triples containing $u$ contribute to $\quinv(\sigma)$ if and only if they also contribute to $\quinv(\Phi_u(\sigma))$. Observe that the degenerate triples containing $u$ are precisely the elements in $w$, the degenerate word of $\sigma$. The reader can easily check that $\quinv(w)$ is precisely equal to the contribution to $\quinv(\sigma)$ from the degenerate triples in the degenerate word $w$. Therefore let us write $\quinv(\sigma)=\quinv'(\sigma)+\quinv(w)$, where $\quinv'(\sigma)$ is the contribution to $\quinv(\sigma)$ coming from all triples but the degenerate triples in $w$. We have $\quinv'(\Phi_u(\sigma))=\quinv'(\sigma)$ by \cref{lem:Phi_u:quinv}, and $\quinv'(\tau_j(\sigma))=\quinv'(\sigma)$ by \cref{lem:tau:quinv} for all $\tau_j$ acting on columns that contain the degenerate word $w$. By construction we have that $\Theta(\sigma)$ has degenerate word $\phi(w)$, and since $\quinv(w)=\quinv(\phi(w))$ by \cref{thm:invol}, we have thus $\quinv(\Theta(\sigma))=\quinv(\sigma)$, which completes our argument.

\emph{$\Theta$ satisfies (iv):} The content of $\sigma$ is changed only by the operator $\Phi_{u(\sigma)}$, which flips exactly one entry from positive to negative, thus reducing $p(\sigma)$ by 1, while leaving the absolute value of the content of $\sigma$ unchanged.
\end{proof}

\cref{prop:Phi-degenerate} follows as an immediate corollary.

\begin{example}
Consider 
$\sigma_1$ 
from \cref{ex:u_choice}. The degenerate word is $w=21\bar1\in W^{>}$, $\phi(w)=\bar12\bar1$, and $\widehat{w}=2\bar1\bar1$. Thus the permutation to get from $\widehat{w}$ to $\phi(w)$ is $\pi=(213)$ which has corresponding PDS $\pi=s_1$, and $u(\sigma)$ is the first cell with content 1 in reading order. Thus we apply $\tau_1$ to $\sigma_1$ after flipping the 1 to obtain 
\[
\Theta(\sigma_1)=\tau_1\circ\Phi_{u(\sigma)}(\sigma)=\raisebox{0.2in}{\tableau{2\\1&\bar1&2&\bar1\\3&\bar2&1&\bar1\\2&1&2&1&\bar2}}\in W^{\leq}.
\]

Starting with 
\[
\sigma_2=
\raisebox{0.2in}{\tableau{2\\1&1&2&1\\3&\bar2&1&\bar1\\2&1&2&1&\bar2}}
\in W^{>},
\]
$w = 121, \phi(w) = \bar{1}12, \hat{w} = 12\bar{1}$, so that the permutation $\hat{w}\pi=\phi(w)$ acting on the right is $\pi=s_2 s_1$. The number of columns left of the degenerate segment is $\ell=1$, and the cell $u(\sigma_2)=(3,4)$. Thus $\Theta(\sigma_2)=\tau_2\circ\tau_3\circ\Phi_{(3,4)}(\sigma_2)$:
\[
\sigma_2  \xrightarrow{\Phi_{(3,4)}}
\raisebox{0.2in}{\tableau{2\\1&1&2&\bar1\\3&\bar2&1&\bar1\\2&1&2&1&\bar2}}  \xrightarrow{\tau_3}
\raisebox{0.2in}{\tableau{2\\1&1&\bar1&2\\3&\bar2&\bar1&1\\2&1&2&1&\bar2}}  \xrightarrow{\tau_2}
\raisebox{0.2in}{\tableau{2\\1&\bar1&1&2\\3&\bar2&\bar1&1\\2&1&2&1&\bar2}}  = \Theta(\sigma_2)\in W^{\leq}.
\]

\end{example}

\section{A bijection on words}
\label{sec:bijwords}

The aim of this section is to prove \cref{thm:invol}

For the proofs, we ask the reader to refer to the terminology and basic results on $q$-series from \cref{sec:qseries}.
For such a map to exist, the $\quinv$ generating function of $W$ must be a multiple
of 2. To see this, recall the definition of $\quinv$ from \eqref{word-quinv-def} and write
\begin{align*}
\sum_{w \in W} q^{\quinv(w)} =& \sum_{k=0}^L \sum_{w \in W_k} q^{\coinv(w) + \binom{\alpha_\ell}{2} + \cdots + \binom{\alpha_{n-1}}{2} + \binom{k}{2} } \\
=&\ q^{\binom{\alpha_\ell}{2} + \cdots + \binom{\alpha_{n-1}}{2}} \sum_{k=0}^L q^{\binom{k}{2}} \qbinom N{L-k,\alpha_2,\dots, \alpha_{n-1},k} \\
=&\ q^{\binom{\alpha_\ell}{2} + \cdots + \binom{\alpha_{n-1}}{2}} \qbinom{N}{L,\alpha_2,\dots,\alpha_{n-1}} \sum_{k=0}^L q^{\binom{k}{2}} \qbinom Lk,
\end{align*}
where we have used \cref{prop:coinv-gf} in the second line. Now, the sum is a direct application 
of the $q$-binomial theorem, \cref{prop:qbinom}, with $x=1$ and we find that
\[
\sum_{w \in W} q^{\quinv(w)} = 2 q^{\binom{\alpha_\ell}{2} + \cdots + \binom{\alpha_{n-1}}{2}}  \qbinom{N}{L,\alpha_2,\dots,\alpha_{n-1}} \prod_{j=1}^{L-1} (1 + q^j),
\]
as claimed.

The following theorem immediately implies \cref{thm:invol}.

\begin{theorem}
\label{thm:quinv-gf}
For $n, N, L, \alpha$, $k$ and $\ell$ as above, we have the identities
\begin{align}
\label{quinv-gf>}
\sum_{w \in W_k^>} q^{\quinv'(w)} =& q^{\binom{\alpha_\ell}{2} + \cdots + \binom{\alpha_{n-1}}{2} + \binom{k+1}{2}} \qbinom{N}{L,\alpha_2,\dots,\alpha_{n-1}} \qbinom{L-1}{k}, \\
\label{quinv-gf<}
\sum_{w \in W_k^\leq} q^{\quinv'(w)} =& q^{\binom{\alpha_\ell}{2} + \cdots + \binom{\alpha_{n-1}}{2} + \binom{k}{2}} \qbinom{N}{L,\alpha_2,\dots,\alpha_{n-1}} \qbinom{L-1}{k-1}.
\end{align}
\end{theorem}

\begin{proof}[Proof of \cref{thm:invol}]
As a consequence of \cref{thm:quinv-gf}, the $\quinv'$ generating functions of $W_{k+1}^>$ and $W_k^\leq$ are equal. The strategy of the proof involves identifying the letters $2, \dots, n-1$ and is clearly independent of the subword.
\end{proof}

\begin{example}
\label{eg:W<>}
As an illustration of \cref{thm:quinv-gf}, consider \cref{eg:Wk} and let $\ell = 3$. Then $W_1^\leq = \{132, 312, 321 \}$ and $W_1^> = \{123, 213, 231 \}$. One can check that the quinv generating functions of both $W_0$ and $W_1^\leq$ are $1 + q + q^2$, and those of both $W_1^>$ and $W_2$ are $q + q^2 + q^3$. Since the coefficients of all powers of $q$ are $1$, the bijection $\phi$ claimed in \cref{thm:invol} is unique in this case.
\end{example}

Although \cref{thm:quinv-gf} looks as if it should be known, we have not seen this in the literature before. The result follows from a more general result, which we now state. For $w \in W_k$, let $p_n(w)$ be the position of the leftmost $n$ in $w$ from the left and $p_1(w)$ be the position of the rightmost $1$ from the right ignoring all $n$'s.

\begin{theorem}
\label{thm:quinv-refined-gf}
Let $N, L, \alpha$, $k$, and $\ell$ be as above. 
Then
\begin{align}
\label{quinv-ref-gf>}
\sum_{\substack{ w \in W_k^> \\ p_1(w) = i}} q^{\quinv'(w)} =& q^{\binom{\alpha_\ell}{2} + \cdots + \binom{\alpha_{n-1}}{2} + \binom{k+1}{2} + (i-1)L} 
\qbinom{N-L}{\alpha_2,\dots,\alpha_{n-1}} \cr 
& \times \qbinom{N-i}{k, N-L-i+1, L-k-1}, \\
\label{quinv-ref-gf<}
\sum_{\substack{ w \in W_k^\leq \\ p_n(w) = j}} q^{\quinv'(w)} =& q^{\binom{\alpha_\ell}{2} + \cdots + \binom{\alpha_{n-1}}{2}+\binom{k}{2} + (j-1) L} 
\qbinom{N-L}{\alpha_2,\dots,\alpha_{n-1}} \cr 
& \times \qbinom{N-j}{k-1, N-L-j+1, L-k}.
\end{align} \end{theorem}

These two identities imply the corresponding ones in \cref{thm:quinv-gf}.

\begin{proof}[Proof of \cref{thm:quinv-gf}]
We use the telescoping sum identity in \cref{prop:q-telescope} to sum \eqref{quinv-ref-gf>} and \eqref{quinv-ref-gf<} over $i$ and $j$, to obtain the identities \eqref{quinv-gf>} and \eqref{quinv-gf<}, respectively.
\end{proof} 
We prove \cref{thm:quinv-refined-gf} using the lemmata stated after, which are in turn
proved in the following two subsections. 
\begin{proof}[Proof of \cref{thm:quinv-refined-gf}]
It will suffice to prove \cref{thm:quinv-refined-gf} for the case of $n=3$. 
To see this, let $w = (w_1, \dots, w_N) \in W_k$ and $w'$ be obtained from $w$ by replacing all
occurences of the letters from $2$ through $n-1$ by $2$, and all the occurences of $n$ by $3$.
Let $w''$ be the subword of $w$ consisting of the letters from $2$ through $n-1$ in $w$. 
Then, it is clear that $\coinv(w) = \coinv(w') + \coinv(w'')$.  Moreover, the question of whether $w$
belongs to $W_k^\leq$ or $W_k^>$ is the same as that of $w'$. 
We can then rewrite \eqref{quinv-ref-gf>} as
\[
\sum_{\substack{ w \in W_k^> \\ p_1(w) = i}} q^{\quinv'(w)} =
q^{\binom{\alpha_\ell}{2} + \cdots + \binom{\alpha_{n-1}}{2} + \binom{k+1}{2} + (i-1)L} 
\qbinom{N-L}{\alpha_2,\dots,\alpha_{n-1}}
\sum_{w'\in V_k^{>}} q^{\coinv(w')},
\]
where $V_k^> = \big(W^{(N-L),L}_k \big)^>$.

We now restrict to the case $n=3$ and $\alpha = (\alpha_2) = (N-L)$. 
We first consider \eqref{quinv-ref-gf>}.
By summing the two cases in \cref{lem:quinv>}, we see that, for $1 \leq i \leq N +1 - L$,
\begin{align*}
\sum_{\substack{ w \in W_k^> \\ p_1(w) = i}} q^{\coinv(w)} 
=&\ q^{ iL + k - L} \qbinom{N-k-i}{L-k-1} \qbinom {N-i}k \\
=&\ q^{ iL + k - L}  \qbinom{L-1}k \qbinom{N-i}{L-1}.
\end{align*}
Now, the $i$-sum can be performed using \cref{prop:q-telescope}. We then simplify to obtain the result.
For \eqref{quinv-ref-gf<}, the argument is similar. Summing the two cases in \cref{lem:quinv<}
shows that, for $1 \leq j \leq N +1 - L$,
\begin{align*}
\sum_{\substack{ w \in W_k^\leq \\ p_3(w) = j}} q^{\coinv(w)} 
=&\ q^{ j L - L} \qbinom{N-j-k+1}{L-k} \qbinom {N-j}{k-1} \\
=&\ q^{ j L - L} \qbinom{L-1}{k-1} \qbinom{N-j}{L-1}.
\end{align*}
Again, the $j$-sum can be performed using \cref{prop:q-telescope}. This completes the proof.
\end{proof}

Let $p'_1(w)$ be the actual position of the rightmost $1$ from the right in $w$.
For both lemmata, there are two subcases to consider, either $p_3(w) < N+1 - p'_1(w)$ or not.
That is to say, the leftmost 3 is to the left of the rightmost 1 or not. The first lemma is for the subset $W_k^>$.

\begin{lemma}
\label{lem:quinv>}
Let $n=3$, and $N, L$ and $k$ be as above. Then, 
\begin{enumerate}
\item
For $1 \leq i \leq \lfloor (N-k)/2 \rfloor$,
\begin{multline}
\label{quinv>l}
\sum_{\substack{ w \in W_k^> \\ p_1(w) = i \\ p_3(w) < N+1 - p'_1(w)}} q^{\coinv(w)} = q^{ (i-1) L} 
\qbinom{N-k-i}{L-k-1} \Bigg( \qbinom {N-i}k q^k \\
 - \qbinom{k+i-1}k q^{k(N-k-2i+2)} \Bigg).
\end{multline}

\item 
For $1 \leq i \leq \lfloor (N-k)/2 \rfloor$,
\begin{equation}
\label{quinv>r1}
\sum_{\substack{ w \in W_k^> \\ p_1(w) = i \\ p_3(w) > N+1 - p'_1(w)}} q^{\coinv(w)} = q^{ (i-1) L + k(N-k-2i+2)} 
\qbinom{N-k-i}{L-k-1} \qbinom{k+i-1}k,
\end{equation}
and for $\lfloor (N-k)/2 \rfloor + 1 \leq i \leq N +1 - L$,
\begin{equation}
\label{quinv>r2}
\sum_{\substack{ w \in W_k^> \\ p_1(w) = i \\ p_3(w) > N+1 - p'_1(w)}} q^{\coinv(w)} = q^{ (i-1) L + k} 
\qbinom{N-k-i}{L-k-1} \qbinom {N-i}k.
\end{equation}
\end{enumerate}
\end{lemma}

The second lemma is for the subset $W_k^\leq$.

\begin{lemma}
\label{lem:quinv<}
Let $n=3$, and $N, L$ and $k$ be as above. Then, 
\begin{enumerate}
\item
For $1 \leq j \leq \lfloor (N-k+1)/2 \rfloor$,
\begin{multline}
\label{quinv<l}
\sum_{\substack{ w \in W_k^\leq \\ p_3(w) = j \\ p_3(w) < N+1 - p'_1(w)}} q^{\coinv(w)} 
= \qbinom{N-j}{k-1} \Bigg( \qbinom {N-j-k+1}{L-k} q^{(j-1)L} \\
  - \qbinom{j-1}{L-k} q^{(j-1)k + (L-k)(N-j-k+1)} \Bigg).
\end{multline}

\item 
For $L-k+1 \leq j \leq \lfloor (N-k+1)/2 \rfloor$,
\begin{equation}
\label{quinv<r1}
\sum_{\substack{ w \in W_k^\leq \\ p_3(w) = j \\ p_3(w) > N+1 - p'_1(w)}} q^{\coinv(w)} = q^{(j-1)k + (L-k)(N-j-k+1)}
\qbinom{j-1}{L-k} \qbinom{N-j}{k-1},
\end{equation}
and for $\lfloor (N-k+1)/2 \rfloor + 1 \leq j \leq N +1 - L$,
\begin{equation}
\label{quinv<r2}
\sum_{\substack{ w \in W_k^\leq \\ p_3(w) = j \\ p_3(w) > N+1 - p'_1(w)}} q^{\coinv(w)} = q^{ (j-1) L} 
\qbinom{N-j-k+1}{L-k} \qbinom {N-j}{k-1}.
\end{equation}
\end{enumerate}
\end{lemma}

\subsection{Proof of \cref{lem:quinv>}}

We will now prove the cases of \cref{lem:quinv>} separately.
We first begin with the second case, since it is conceptually simpler.

\begin{proof}[Proof of \cref{lem:quinv>}(2)]
Suppose $w \in  W_k^>$ with $p_3(w) = j$, $p'_1(w) = i+k$ and $j > N+1 - i-k$. Then
$w$ can be written as $w = w_1 1 2^{i+j+k-N-2} 3 w_2$, where $w_1 \in \{1,2\}^{N-i}$ with $L-k-1$ $1$'s
and $w_2 \in \{2,3\}^{N-j}$ with $k-1$ $3$'s. As a result, $p_1(w) = i$ and
\[
\coinv(w) = \coinv(w_1) + \coinv(w_2) + (L-k)(i+k-1) + k(j+k-L-1).
\]
Notice that the smallest (resp. largest) possible value of $i$ is $1$ (resp. $N-L+1$), namely when 
$j = N+2-i-k$ and $w_2$ contains no $2$'s
(resp. $w_1$ contains no $2$'s).
Let the left hand sides of \eqref{quinv>r1} and \eqref{quinv>r2} be denoted $f(i)$.
By definition of $W_k^>$, we have $j > i$. Therefore,
\[
f(i) = \sum_{j=(i+1) \vee (N-i-k+2)}^{N-k+1} \sum_{w_1,w_2} q^{\coinv(w_1) + \coinv(w_2) + (L-k)(i+k-1) + k(j+k-L-1)},
\]
where $a \vee b$ denotes the maximum of $a$ and $b$,
and the sums over words $w_1$ and $w_2$ are performed using \cref{prop:coinv-gf} to give
\[
f(i) = \sum_{j=(i+1) \vee (N-i-k+2)}^{N-k+1} \qbinom {N-i-k}{L-k-1} \qbinom{N-j}{k-1} q^{ (L-k)(i-k-1) + k(j+k-L-1)}.
\]
Now, the $j$-sum can be performed using \cref{prop:q-telescope}. However, the result depends on 
the lower limit. If $1 \leq i \leq \lfloor (N-k)/2 \rfloor$, then the lower limit is $N-i-k+2$, and otherwise,
it is $i-k+1$. These two different cases prove \eqref{quinv>r1} and \eqref{quinv>r2}.
\end{proof}

\begin{proof}[Proof of \cref{lem:quinv>}(1)]
Suppose $w \in  W_k^>$ with $p_3(w) = j$, $p'_1(w) = i+b_3$ and $j < N+1 - i-b_3$. Then
$w$ can be written as $w = w_1 3 w_2 1 w_3$, where $w_1 \in \{1,2\}^{j-1}$, $w_2 \in \{1,2,3\}^{N-i-b_3-j}$ 
and $w_3 \in \{2,3\}^{i+b_3-1}$ with a total of $L-k-1$ $1$'s in $w_1$ and $w_2$ and a total of
$k-1$ $3$'s in $w_2$ and $w_3$. 
Suppose $a_1$ is the number of $1$'s in $w_1$ and $b_3$ is the number of $3$'s in $w_3$.
Then, $p_1(w) = i$ and
\begin{multline*}
\coinv(w) = \coinv(w_1) + \coinv(w_2) +\coinv(w_3) + (L-k-a_1)(i+b_3-1)  \\
+ k(j-1) + a_1(N-j-L+1+a_1) + b_3(N-L-i-j+a_1+2).
\end{multline*}
Let the left hand side of \eqref{quinv>l} be denoted $f(i)$.
We now follow the strategy of the proof of \cref{lem:quinv>}(2). The sums over all
words $w_1, w_2, w_3$ will give us appropriate $q$-multinomial coefficients by 
\cref{prop:coinv-gf}. Then we will obtain
\begin{multline*}
f(i) =  \sum_{b_3 = 0}^{(k-1) \wedge (N-i)}   \sum_{j=i+1}^{N-i-b_3}  \sum_{a_1 = 0}^{(j-1) \wedge (L-k-1)} 
\qbinom {j-1}{a_1} \qbinom{i+b_3-1}{b_3}\\
\times \qbinom {N-i-j-b_3}{L-k-a_1-1, N-i-j-L+a_1+2,k-1-b_3} \\
\times q^{(L-k-a_1)(i+b_3-1) + k(j-1) + a_1(N-j-L+1+a_1) + b_3(N-L-i-j+a_1+2)},
\end{multline*}
where $a \wedge b$ means the minimum of $a$ and $b$.
Let us clarify the limits of the sums first. 
The $b_3$-sum is bounded above by the minimum
of the available number of $3$'s and the fact that $w_2$ is of nonnegative length. 
Now, $j$ is at least $i+1$ by definition of the set 
$W_k^>$ and is at most $N-i-b_3$ again because of $w_2$.
Then, $a_1$ is at most the minimum of the available number of $1$'s and the length
of $w_1$.
Since both upper limits of the $b_3$-sum and the $a_1$-sum are natural, the answer
is independent of which of them is larger. Therefore, we will only write one of them.

The $a_1$-sum can now be performed using \cref{thm:q-chuvan} and we end up with
\begin{multline*}
f(i) = \sum_{b_3 = 0}^{k-1}   \sum_{j=i+1}^{N-i-b_3} q^{(L-k)(i+b_3-1) + k(j-1) + b_3(N-L-i-j+2)} \qbinom{i+b_3-1}{b_3} \\
\times \qbinom {N-i-b_3-j}{k-1-b_3}  \qbinom {N-i-k}{L-k-1} .
\end{multline*}
We now find that the $j$-sum can be performed using \cref{prop:q-telescope}.  The result is then
\[
f(i) = \sum_{b_3 = 0}^{k-1}  q^{(i-1)L + b_3(N-k-2i+1) + k} 
\qbinom{i+b_3-1}{i-1} \qbinom {N-2i-b_3}{N-k-2i} 
\qbinom {N-i-k}{L-k-1} .
\]
Now, notice that if the upper limit of the $b_3$-sum were $k$, we could have applied
\cref{cor:q-dualchu} after replacing $b_3$ by $b_3+i-1$.  We thus add and subtract 
the term for $b_3 = k$, perform the $b_3$-sum and simplify to obtain the result.
\end{proof}

\subsection{Proof of \cref{lem:quinv<}}
We will now prove the cases of \cref{lem:quinv<} separately.
We again begin with the second case first, since it is conceptually simpler.

\begin{proof}[Proof of \cref{lem:quinv<}(2)]
Suppose $w \in  W_k^\leq$ with $p_3(w) = j$, $p'_1(w) = i+k$ and $j \geq N+1 - i-k$. Then
$w$ can be written as $w = w_1 1 2^{i+j+k-N-2} 3 w_2$, where $w_1 \in \{1,2\}^{N-i-k}$ with $L-k-1$ $1$'s
and $w_2 \in \{2,3\}^{N-j}$ with $k-1$ $3$'s. As a result, $p_1(w) = i$ and
\[
\coinv(w) = \coinv(w_1) + \coinv(w_2) + (L-k)(i+k-1) + k(j+k-L-1).
\]
Notice that the smallest possible value of $j$ is $L-k+1$, namely when 
$i = N+2-j-k$ and $w_1$ contains no $2$'s
To see the  largest possible value of $j$,  note that $j \leq i$ by definition of $W_k^\leq$ 
and the maximum value of $i$ is $N-L+1$ when $w_1$ contains no $2$'s.
Let the left hand sides of \eqref{quinv<r1} and \eqref{quinv<r2} be denoted $f(j)$.
Therefore,
\[
f(j) = \sum_{i=j \vee (N-j-k+2)}^{N+k+1-L}\ \sum_{w_1,w_2} q^{\coinv(w_1) + \coinv(w_2) + (L-k)(i-1) + k(j-1)},
\]
and the sums over words $w_1$ and $w_2$ are performed using \cref{prop:coinv-gf} to give
\[
f(j) = \sum_{i=j \vee (N-j-k+2)}^{N+k+1-L} \qbinom {N-i-k}{L-k-1} \qbinom{N-j}{k-1} q^{ (L-k)(i-1) + k(j-1)}.
\]
Now, the $i$-sum can be performed using \cref{prop:q-telescope}. However, the result depends on 
the lower limit. If $L-k+1 \leq j \leq \lfloor (N-k+1)/2 \rfloor$, then the lower limit is $N-j-k+2$, and otherwise,
it is $j$. These two different cases prove \eqref{quinv<r1} and \eqref{quinv<r2}.
\end{proof}

\begin{proof}[Proof of \cref{lem:quinv<}(1)]
Suppose $w \in  W_k^\leq$ with $p_3(w) = j$, $p'_1(w) = i+b_3$ and $j < N+1 - i-b_3$. Then
$w$ can be written as $w = w_1 3 w_2 1 w_3$, where $w_1 \in \{1,2\}^{j-1}$, $w_2 \in \{1,2,3\}^{N-i-b_3-j}$ 
and $w_3 \in \{2,3\}^{i+b_3-1}$ with a total of $L-k-1$ $1$'s in $w_1$ and $w_2$, and a total of
$k-1$ $3$'s in $w_2$ and $w_3$. 
Suppose $a_1$ is the number of $1$'s in $w_1$ and $b_3$ is the number of $3$'s in $w_3$.
Then, $p_1(w) = i$ and
\begin{multline*}
\coinv(w) = \coinv(w_1) + \coinv(w_2) +\coinv(w_3) + (L-k-a_1)(i+b_3-1)  \\
+ k(j-1) + a_1(N-j-L+1+a_1) + b_3(N-L-i-j+a_1+2).
\end{multline*}
Let the left hand side of \eqref{quinv>l} be denoted $f(j)$.
We now follow the strategy of the proof of \cref{lem:quinv<}(2). The sums over all
words $w_1, w_2, w_3$ will give us appropriate $q$-multinomial coefficients by 
\cref{prop:coinv-gf}. Then we will obtain
\begin{multline*}
f(j) = \sum_{i = j}^{N-k-j+1}\  \sum_{b_3 = 0}^{(k-1) \wedge (N-i-j)}\ \sum_{a_1 = 0}^{(j-1) \wedge (L-k-1)} 
\qbinom {j-1}{a_1} \qbinom{i+b_3-1}{b_3} \\
\times\qbinom {N-i-j-b_3}{L-k-a_1-1, N-i-j-L+a_1+2,k-1-b_3} \\
\times q^{(L-k-a_1)(i+b_3-1) + k(j-1) + a_1(N-j-L+1+a_1) + b_3(N-L-i-j+a_1+2)}.
\end{multline*}
Again, we need to clarify the limits. The lower limit of $i$ is $j$ by definition of $W_k^\leq$
and the upper limit of $i$ comes from the fact that the number of $3$'s in $w_2$ is upper bounded
by $N-i-j-b_3$. The upper limits of $a_1$ and $b_3$ are explained by the same reasoning
as the similar sum in the proof of \cref{lem:quinv>}(1).

The $a_1$-sum can now be performed using \cref{thm:q-chuvan} and we end up with
\begin{multline*}
f(j) =  \sum_{i = j}^{N-k-j+1}  \sum_{b_3 = 0}^{k-1}  q^{(L-k)(i+b_3-1) + k(j-1) + b_3(N-L-i-j+2)} \qbinom{i+b_3-1}{b_3} \\
\times \qbinom {N-i-b_3-j}{k-1-b_3}  \qbinom {N-i-k}{L-k-1} .
\end{multline*}
We now find that the $b_3$-sum can be performed using \cref{cor:q-dualchu} after replacing $b_3$
by $b_3+i-1$.  
The result is then
\[
f(j) = \sum_{i = j}^{N-k-j+1}  q^{(i-1)(L-k) + k(j-1)}  \qbinom {N-j}{k-1}\qbinom {N-i-k}{L-k-1} ,
\]
We now rewrite the limits of the $i$-sum as $\sum_{i = j}^{N-L+1} - \sum_{i = N-k-j+2}^{N-L+1}$,
 use \cref{prop:q-telescope} and simplify to complete the proof.
\end{proof}

\section{Conclusion and further questions}\label{sec:conclusion}

We are finally equipped to prove the main result of this paper, \cref{thm:mainconj}, stating that $\widetilde{H}_{\lambda}(X;q,t) = C_{\lambda}(X;q,t)$ where 
\[
C_{\lambda}(X;q,t) = \sum_{\sigma\in\PQT(\lambda)}x^{\sigma}t^{\quinv(\sigma)}q^{\maj(\sigma)}.
\]

\begin{proof}[Proof of \cref{thm:mainconj}]
By \cref{thm:symmetry}, $C_{\lambda}(X;q,t)$ is symmetric in the variables $x_i$. Axioms \eqref{eq:A1}, \eqref{eq:A2}, and \eqref{eq:A3} uniquely characterize and define the modified Macdonald polynomials $\widetilde{H}_{\lambda}(X;q,t)$. $C_{\lambda}(X;q,t)$ satisfies Axiom \eqref{eq:A3} by definition. Axioms \eqref{eq:A1} and \eqref{eq:A2} are written equivalently as \eqref{A1m} and \eqref{A2m}, respectively. Those can in turn be rewritten in terms of $\widetilde{C}_{\lambda}(X;q,t)$, the superization of $C_{\lambda}(X;q,t)$, as \eqref{C1} and \eqref{C2}, respectively. \eqref{C1} is true by \cref{lem:A1FP}. \eqref{C2} is true by \cref{thm:A2FP} together with \cref{prop:Phi-degenerate}. This completes the proof. 
\end{proof}

Our work leads to several natural questions, which we mention below.

\begin{question}
From (\ref{eq:HHLformula}) and (\ref{eq:MainResult}), it follows that the modified Macdonald polynomial $\widetilde{H}_\lambda(X; q, t)$ can be expressed as a sum over tableaux either of queue inversion weights or of HHL weights. 
In these weights, the terms for the content and the major index are the same. The only difference is in the notion of a ``triple'', as discussed at the end of Section \ref{sec:definitions}. Therefore it is natural to ask for a bijective proof of the equality of these sums.

We conjecture something stronger: there is a bijection from $\PQT(\lambda, n)$ 
to itself that preserves the row content of the fillings and the major index, and sends the $\quinv$ statistic to the HHL $\inv$ statistic. To formally state this, we define row-equivalency classes for the fillings.

\begin{defn}
Let $\sigma,\tau$ be fillings of $\dg(\lambda)$. We say $\sigma$ and $\tau$ are \emph{row-equivalent} if for every row of $\sigma$, its entries are a permutation of the entries of the corresponding row in $\tau$. We write $\sigma \sim \tau$ when $\sigma$ and $\tau$ are row-equivalent, and we denote by $[\sigma]$ the class of row-equivalent fillings that $\sigma$ belongs to. 
\end{defn}

\begin{conjecture}
Let $[\tau]$ be a row-equivalency class. Then
\[
\sum_{\sigma\in[\tau]} t^{\quinv(\sigma)}q^{\maj(\sigma)} = \sum_{\sigma\in[\tau]} t^{\inv(\sigma)}q^{\maj(\sigma)}
\]
\end{conjecture}

The following conjecture is sufficient to prove \cref{thm:mainconj} via \cite[Theorem 2.2]{HHL05}, since it gives a weight-preserving bijection from queue inversion weights to HHL weights.

\begin{conjecture}\label{conj:bij}
Given $\lambda$ and $n$, there exists a bijection $\delta:\PQT(\lambda,n) \rightarrow \PQT(\lambda,n)$ such that for all $\sigma$,
\begin{enumerate}[label=\roman*.]
\item $\delta(\sigma) \sim \sigma$,
\item $\maj(\delta(\sigma)) = \maj(\sigma)$,
\item $\inv(\delta(\sigma)) = \quinv(\sigma)$.
\end{enumerate} 
\end{conjecture}

\begin{example}
Consider the row-equivalency class $[T]=(\{1\},\{2,2,3\},\{1,1,2\})$ for $\lambda=(3,2,2)$. In \cref{fig:rowequiv}, we show the tableaux $F\in[T]$ with their queue inversion weights $t^{\quinv(F)}q^{\maj(F)}$ above
above and HHL weights $t^{\inv(F)}q^{\maj(F)}$ below. Observe that the sums of the weights are equal. Moreover, there is a bijection $\delta$  such that the queue inversion weight of $F$ matches the HHL weight of $\delta(F)$.
\end{example}

\begin{figure}[h!]
\begin{center}
\resizebox{\textwidth}{!}{
\begin{tikzpicture}[node distance=2cm]
\def \sh {1.6};
\node at (-1.2,-1.25) {quinv};
\node at (-1.2,-2) {HHL};

\node at (0,0) {$\tableau{1\\2&2&3\\1&1&2}$};
\node at (0,-1.25) {$t^2q^4$};
\node at (0,-2) {$q^4$};

\node at (\sh,0) {$\tableau{1\\2&2&3\\1&2&1}$};
\node at (\sh,-1.25) {$t^3q^3$};
\node at (\sh,-2) {$t^2q^3$};

\node at (2*\sh,0) {$\tableau{1\\2&2&3\\2&1&1}$};
\node at (2*\sh,-1.25) {$t^4q^2$};
\node at (2*\sh,-2) {$t^4q^2$};

\node at (3*\sh,0) {$\tableau{1\\2&3&2\\1&1&2}$};
\node at (3*\sh,-1.25) {$t^2q^3$};
\node at (3*\sh,-2) {$tq^3$};

\node at (4*\sh,0) {$\tableau{1\\2&3&2\\1&2&1}$};
\node at (4*\sh,-1.25) {$tq^4$};
\node at (4*\sh,-2) {$tq^4$};

\node at (5*\sh,0) {$\tableau{1\\2&3&2\\2&1&1}$};
\node at (5*\sh,-1.25) {$t^3q^2$};
\node at (5*\sh,-2) {$t^3q^2$};

\node at (6*\sh,0) {$\tableau{1\\3&2&2\\1&1&2}$};
\node at (6*\sh,-1.25) {$tq^3$};
\node at (6*\sh,-2) {$t^2q^3$};

\node at (7*\sh,0) {$\tableau{1\\3&2&2\\1&2&1}$};
\node at (7*\sh,-1.25) {$t^2q^3$};
\node at (7*\sh,-2) {$t^3q^3$};

\node at (8*\sh,0) {$\tableau{1\\3&2&2\\2&1&1}$};
\node at (8*\sh,-1.25) {$q^4$};
\node at (8*\sh,-2) {$t^2q^4$};

\end{tikzpicture}
}
\end{center}
\caption{For $\lambda=(3,2,2)$ and $n=3$, all fillings of type $(\lambda,n)$ in the row-equivalency class $[T]=(\{1\},\{2,2,3\},\{1,1,2\})$. Below are the corresponding queue 
inversion weights and HHL weights.}\label{fig:rowequiv}
\end{figure}
\end{question}

\begin{question}
In \cref{sec:bijwords}, we demonstrated in \cref{thm:invol} the existence of a quinv-preserving bijection between sets of words respecting certain conditions. 
We have not been able to find such a bijection and we think it would be interesting to find one.
\end{question}

\section*{Acknowledgements}
This work was inspired by OM's discussion with Lauren Williams and Sylvie Corteel, in particular Lauren Williams' suggestion of connecting modified Macdonald polynomials to a particle process via tableaux and the queue inversion statistic. OM also gratefully acknowledges conversations with Jim Haglund and Sarah Mason. This material is based upon work supported by the Swedish Research Council under grant no. 2016-06596 while the authors were in residence at Institut Mittag-Leffler in Djursholm, Sweden during the spring semester of 2020. The first author (AA) was partially supported by the UGC Centre for Advanced Studies and by Department of Science and Technology grant
EMR/2016/006624.
The second author (OM) was supported by NSF grants DMS-1704874 and DMS-1953891.

\bibliographystyle{alpha}
\bibliography{Macbib}


\end{document}